\documentclass[12pt,article]{amsart}
\usepackage{amsmath}
\usepackage{amssymb}
\newtheorem{corollary}[equation]{Corollary}

\newtheorem{claim}[equation]{\indent{\it Claim}\rm }
\newtheorem{example}[equation]{\indent \rm {\it Example}}

\newtheorem{lemma}[equation]{Lemma}
\newtheorem{proposition}[equation]{Proposition}

\newtheorem{theorem}[equation]{Theorem}

\newcommand{\C}{{\mathbb{C}}}

\renewcommand{\P}{{\mathbb{P}}}
\newcommand{\E}{{\mathbb{E}}}

\newcommand{\R}{{\mathbb{R}}}

\newcommand{\rank}{\mathrm{rank}}

\newcommand{\supp}{\mathrm{Supp}\,}

\numberwithin{equation}{section}

\title[Modified defect relation for Gauss maps of minimal surfaces]{Modified defect relation for Gauss maps of minimal surfaces with hypersurfaces of projective varieties\\ in subgeneral position} 

\author{Si Duc Quang}

\address{$^1$Department of Mathematics, Hanoi National University of Education\\
136-Xuan Thuy, Cau Giay, Hanoi, Vietnam}
\address{$^2$Institute of Natural Sciences, Hanoi National University of Education\\
136-Xuan Thuy, Cau Giay, Hanoi, Vietnam}
\email{quangsd@hnue.edu.vn}

\begin{document}

\begin{abstract}
In this paper, we establish some modified defect relations for the Gauss map $g$ of a complete minimal surface $S\subset\mathbb R^m$ into a $k$-dimension projective subvariety $V\subset\mathbb P^n(\mathbb C)\ (n=m-1)$ with hypersurfaces $Q_1,\ldots,Q_q$ of $\mathbb P^n(\mathbb C)$ in $N$-subgeneral position with respect to $V\ (N\ge k)$. In particular, we give the upper bound for the number $q$ if the image $g(S)$ intersects each hypersurfaces $Q_1,\ldots,Q_q$ a finite number of times and $g$ is nondegenerate over $I_d(V)$, where $d=lcm(\deg Q_1,\ldots,\deg Q_q)$, i.e., the image of $g$ is not contained in any hypersurface $Q$ of degree $d$ with $V\not\subset Q$. Our results extend and generalize the previous results for the case of the Gauss map and hyperplanes in a projective space. The results and the method of this paper have been applied by some authors to study the unicity problem of the Gauss maps sharing families of hypersurfaces.
\end{abstract}

\maketitle

\def\thefootnote{\empty}


\section{Introduction and Main results} 
In a series of papers \cite{O59,O63,O64,O86}, Osserman began studying the value distribution of the Gauss maps of minimal surfaces. We summarize his results as the following theorem.

\vskip0.2cm
\noindent
{\bf Theorem A}\ (R. Osserman).\ {\it Let $S$ be a complete minimal surface in $\R^3$. Then: 
\begin{itemize}
\item[(1)] $S$ has infinite total curvature $\Leftrightarrow$ the Gauss map of $S$ takes on all directions infinitely often with the exception of at most a set of logarithmic capacity zero,
\item[(2)] $S$ has finite non-zero total curvature $\Leftrightarrow$ the Gauss map of $S$ takes on all directions a finite number of times, omitting at most three directions,
\item[(3)] $S$ has zero total curvature $\Leftrightarrow$ $S$ is a plane.
\end{itemize}}

In 1981, Xavier \cite{X81} improved the assertion (1) of Theorem A by showing that the Gauss map of a complete minimal surface $S$ in $\R^3$ can omit at most six directions unless it is a plane. Later on, Fujimoto \cite{Fu88} showed that the result of Xavier still valid if the Gauss map omits only four directions, and the number four is the best possible number.

Combining the early method of Osserman with the method of H. Fujimoto, X. Mo and R. Osserman \cite{MO90} proved the following theorem.

\vskip0.2cm
\noindent
{\bf Theorem B}\ (X. Mo and R. Osserman \cite{MO90}).\ {\it Let $S$ be a complete minimal surface in $\R^3$ with infinite total curvature. Then its Gauss map must take every direction infinitely often except at most four directions.} 

For the case of higher dimension, H. Fujimoto proved the following theorem.

\vskip0.2cm
\noindent
{\bf Theorem C}\ (H. Fujimoto \cite{Fu90}).\ {\it Let $S$ be a complete minimal surface in $\R^m$ with nondegenerate Gauss map. Then the image of $S$ under the Gauss map can omit at most $m(m+1)/2$ hyperplanes in general position in $\P^{m-1}(\C)$.}

And then, X. Mo in \cite{M94} improved the above theorem of H. Fujimoto by replacing the condition ``can omit at most $m(m+1)/2$ hyperplanes'' by the condition ``intersects only a finite number of times with these hyperplanes''.

However, in all the above mentioned results, the assumptions imply directly that there is a compact subset $K$ of $S$ such that the image of $S\setminus K$ under the Gauss map does not intersect these hyperplanes, i.e., the Gauss map at near ``the boundary of $S$'' will omit these hyperplanes. In order to study the general case where such compact subset $K$ may not exist, H. Fujimoto have defined the notion of modified defect for holomorphic curves from $S$ into $\P^{m-1}(\C)$ and proved some modified defect relation for such curves with hyperplanes (cf. \cite{Fu89,Fu90,Fu91}). Here a modified defect relation is an above bound estimate of the sum of some modified defects for the curve with hyperplanes. In those works, H. Fujimoto constructed a metric of negative curvature under certain conditions on $S$ and then using Schwarz-Pick's Lemma to derive an inequality which is the essential key to the study of the value distribution of the Gauss map. Based on that method, X. Mo \cite{M94} proved the following.

\vskip0.2cm
\noindent
{\bf Theorem D}\ (X. Mo \cite{M94}).\ {\it Let $S$ be a complete nondegenerate minimal surface in $\R^m$ such that the Gauss map $f=(f_0:\cdots:f_n)$ (here $n=m-1$) intersects only a finite number of times the hyperplanes $H_1,\ldots,H_q$ of $\P^n(\C)$ in general position. If $q> m(m+1)/2 = (n+1)+n(n+1)/2$, then $S$ must have finite total curvature.}

Our aim in this paper is to generalize the notion of modified defect for holomorphic curves and hyperplanes of a projective space of H. Fujimoto to the case of hypersurfaces in projective varieties, and then study the relations between those modified defect relations for the Gauss maps and hypersurfaces of a projective variety in subgeneral position with the total curvature of minimal surfaces in $\R^m$. In order to do so, firstly we have to establish the defect relation for holomorphic curves from punctured disks into projective varieties with hypersurfaces in subgeneral position. Secondly, we reconstructed all notions, functions and estimates introduced by H. Fujimoto in \cite{Fu89,Fu90,Fu91} for the case of hypersurfaces. To state our results, we introduce the following notions.

Let $S$ be an open complete Riemann surface. In this paper, by a divisor on $S$ we mean a map $\nu$ of $S$ into $\mathbb Z$ whose support $\supp(\nu) :=\{p\in S\ |\ \nu(p)\ne 0\}$ has no accumulation points in $S$. 

Let $f:S\rightarrow\P^n(\C)$ be a holomorphic curve and let $Q$ be a hypersurface in $\P^n(\C)$ of degree $d$. By $\nu_{Q(f)}$ we denote the pullback divisor $Q$ by $f$.  Let $F=(f_0,\ldots,f_n)$ be a global reduced representation of $f$ and $K$ be a compact subset of $S$. Let $A=S\setminus K$. 

The $S$-defect truncated to level $m$ of the hyperpsurface $Q$ for $f$ is defined by 
$$\delta^{S,m}_{f,A}(Q):=1- \mathrm{inf}\{\eta\ |\ \eta \text{ satisfies Condition }(*)_S\}.$$
Here, condition $(*)_S$ means that there exists a $[-\infty,\infty)$-valued continuous subharmonic function $u\ (\not\equiv -\infty)$ on $A$ satisfying the following conditions:
\begin{itemize}
\item[$(D_1)$] $e^u\le \|F\|^{d\eta}$,
\item [$(D_2)$] for each $\xi\in f^{-1}(Q)$ there exists the limit 
$$\underset{z\rightarrow\xi}\lim (u(z)-\min\{\nu_{Q(f)}(z),m\}\log|z -\xi|) \in[-\infty,\infty),$$
where $z$ is a holomorphic local coordinate around $\xi$.
\end{itemize}

The $H$-defect truncated to level $m$ of $Q$ for $f$ is defined by
$$\delta^{H,m}_{f,A}(Q):=1- \mathrm{inf}\{\eta\ |\ \eta \text{ satisfies Condition }(*)_H\}.$$
Here, condition $(*)_H$ means that there exists a $[-\infty,\infty)$-valued continuous subharmonic function $u\ (\not\equiv\infty)$ on $A$, harmonic on $A\setminus f^{-1}(Q)$ and satisfies conditions $(D_1),(D_2)$.

We say a function $u$ with mild singularities on a domain $D\subset S$ if $u$ is a complex-valued function of the class $\mathcal C^\infty$ on $D$ outside a discrete set $E$ and every point $a\in D$ has a neighborhood $U$ such that for a holomorphic local coordinate $z$ with $z(a)=0$ on $U$ we can write
$$ |u(z)|=|z|^\alpha u^*(z)\prod_{j=1}^q\left|\log\dfrac{1}{|g_j(z)|v_j(z)}\right|^{r_j}$$
on $U\setminus E$ with some real number $\sigma$, non-positive real numbers $r_j$, nonzero holomorphic functions $g_j$ with $g_j(0)=0$ and some positive $\mathcal C^{\infty}$ functions $u^*$ and $v_j$ on $U$, where $0\le q<\infty$.

The modified $H$-defect truncated to level $m$ of $Q$ for $f$  is defined by
$$D^m_{f}(Q) := 1 - \mathrm{inf}\{1-\eta\ |\ \eta \text{ satisfies the condition (*)}.\}$$
Here, the condition (*) means that: there are a compact subset $K$ of $S$, a divisor $\nu$ and a continuous real-valued bounded function $k$ with mild singularities on $S\setminus K$, a positive constant $c$ such that $\nu(z)>c$ for each $z\in\supp (\nu)$ and
$$ [\min\{m,\nu_{Q(f)}\}]+[\nu]=d\eta f^*\Omega+dd^c[\log|k|^2] $$
in the sense of current, where $\Omega$ is the Fubini-Study form of $\P^n(\C)$. Hence, if we set $h=|k|\cdot\|F_z\|^{\eta}$ for a local reduced representation $F_z$ (on a local holomorphic chat $U,z$) of $f$ then $\log h$ is harmonic on $U\setminus (K\cup\supp(\nu_h)\cup f^{-1}(Q))$ and $\nu_h -\min\{m,\nu_{H(f)}\}=v$.

By \cite[Proposition 2.3]{Fu90} and \cite[Remark 5.3]{Fu91} we see that 
\begin{align}\label{new}
0 \le  \delta^{H,m}_{f,A}(Q) \le  \delta^{S,m}_{f,A}(Q)\le  1\text{ and }0 \le \delta^{H,m}_{f,A}(Q)\le D^{m}_{f}(Q)\le 1.
\end{align}

Let $V$ be a complex projective subvariety of $\P^n(\C)$ of dimension $k\ (k\le n)$. Let $d$ be a positive integer. We denote by $I(V)$ the ideal of homogeneous polynomials in $\C [x_0,...,x_n]$ defining $V$ and by $\C [x_0,...,x_n]_d$ the vector space of all homogeneous polynomials in $\C [x_0,...,x_n]$ of degree $d$ including the zero polynomial. Define 
$$I_d(V):=\dfrac{\C [x_0,...,x_n]_d}{I(V)\cap \C [x_0,...,x_n]_d}\text{ and }H_V(d):=\dim I_d(V).$$
Then $H_V(d)$ is called the Hilbert function of $V$. Each element of $I_d(V)$ which is an equivalent class of an element $Q\in H_d,$ will be denoted by $[Q]$. 

Let $Q_1,...,Q_q\ (q\ge k+1)$ be $q$ hypersurfaces in $\P^n(\C)$. The family of hypersurfaces $\{Q_i\}_{i=1}^q$ is said to be in $N$-subgeneral position with respect to $V$ if 
$$ V\cap (\bigcap_{j=1}^{N+1}Q_{i_j})=\varnothing \ \forall\ 1\le i_1<\cdots <i_{N+1}.$$
If  $\{Q_i\}_{i=1}^q$ is in $n$-subgeneral position then we say that it is in \textit{general position} with respect to $V.$ Throughout this paper, sometime we will denote by the same notation $Q$ for the defining homogeneous polynomial of the hypersurface $Q$ if there is no confusion.

Let $f:S\longrightarrow V$ be a holomorphic curve. We say that $f$ is degenerate over $I_d(V)$ if there is $[Q]\in I_d(V)\setminus \{0\}$ such that $Q(F)\equiv 0$ for some local reduced representation $F$ of $f$. Otherwise, we say that $f$ is nondegenerate over $I_d(V)$. It is clear that if $f:S\longrightarrow V$ is algebraically nondegenerate, then $f$ is nondegenerate over $I_d(V)$ for every $d\ge 1.$ Our first main result is stated as follows.
 
\begin{theorem}\label{1.1} 
Let $S$ be a complete minimal surface in $\R^m$. Let $V$ be a projective subvariety of dimension $k$ of $\P^{n}(\C)\ (n=m-1)$. Let $Q_1,\ldots,Q_q$ be hypersurfaces of $\P^n(\C)$ in $N$-subgeneral position with respect to $V$. Let $d$ be the least common multiple of $\deg Q_j\ (1\le j\le q)$, i.e., $d=lcm(\deg Q_1,\ldots,\deg Q_q)$. Assume that the Gauss map $g:S\rightarrow V\subset\P^n(\C)$ is nondegenerate over $I_d(V)$ and
$$\sum_{j=1}^qD^{M}_{g}(Q_j)>\dfrac{(2N-k+1)(M+1)}{k+1}+\dfrac{(2N-k+1)M(M+1)}{2d(k+1)},$$
where $M=H_V(d)-1$. Then $S$ has finite total curvature.
\end{theorem}
By the inequality (\ref{new}), from Theorem \ref{1.1} we immediately have the following corollary about $H$-defect relation.

\begin{corollary}\label{1.2} 
Let $S, V,\{Q_i\}_{i=1}^q,d,M$ be as in Theorem \ref{1.1}. Let $K$ be a compact subset of $S$ and $A=S\setminus K$. Assume that the Gauss map $g:S\rightarrow V\subset\P^n(\C)$ is nondegenerate over $I_d(V)$ and
$$\sum_{j=1}^q\delta^{H,M}_{g,A}(Q_j)>\dfrac{(2N-k+1)(M+1)}{k+1}+\dfrac{(2N-k+1)M(M+1)}{2d(k+1)}.$$
Then $S$ has finite total curvature.
\end{corollary}

We note that, a complete minimal surface has a universal covering biholomorphic to $\C$ or a ball in $\C$. Then, from Theorem 4.2  in \cite{QT}, we have the following $S$-defect relation.

\vskip0.2cm
\noindent
{\bf Theorem E}\ (D. D. Thai and S. D. Quang \cite{QT}).\ {\it Let $S$ be a complete minimal surface in $\R^m$. Let $V,\{Q_i\}_{i=1}^q,d,M$ be as in Theorem \ref{1.1}. Assume that the Gauss map $g:S\rightarrow V\subset\P^n(\C)$ is nondegenerate over $I_d(V)$. Then we have
$$\sum_{j=1}^q\delta^{S,M}_{g,S}(Q_j)\le \dfrac{(2N-k+1)(M+1)}{k+1}+\dfrac{(2N-k+1)M(M+1)}{d(k+1)}.$$}

On the other hand, for the case of minimal surface with finite total curvature and the Gauss map is ramified at least $m_j$ over $Q_j\ (1\le j\le q)$, in a recent work \cite{QTT} the author with D. D. Thai and P. D. Thoan showed a sharper defect relation (cf. \cite[Theorem 1.2]{QTT}) as follows
$$\delta^{S,M}_{g,S}(Q_j)\le \le \dfrac{(2N-k+1)(M+1)}{k+1}+\dfrac{(2N-k+1)M(M+1)}{2d(k+1)}.$$
Without the condition ``ramified over targets'', for the case of linear nondegenerate Gauss map and hyperplanes $\{H_i\}_{i=1}^q$ of $\P^n(\C)$ in $N$-subgeneral position, P. H. Ha \cite{Ha18b} showed that
$$\sum_{j=1}^q\delta^{S,n}_{g,S}(H_j)\le 2N-n+1+\dfrac{(2N-n+1)n}{2}.$$
Our second purpose in this paper is to generalize all above results to the case of hypersurfaces of projective varieties in subgeneral position and without the condition ``ramified over hypersurfaces''. Our second main result is stated as follows.
\begin{theorem}\label{1.3} 
Let $S$ be a complete minimal surface with finite total curvature in $\R^m$. Let $V,\{Q_i\}_{i=1}^q,d,M$ be as in Theorem \ref{1.1}. Assume that the Gauss map $g:S\rightarrow V$ is nondegenerate over $I_d(V)$. Then we have
$$\sum_{j=1}^q\delta^{S,M}_{g,S}(Q_j)\le \dfrac{(2N-k+1)(M+1)}{k+1}+\dfrac{(2N-k+1)M(M+1)}{2d(k+1)}.$$
\end{theorem}

From the above theorems, we get a corollary on the ramification of the Gauss map as follows.
\begin{corollary}\label{1.4}
Let $S, V,\{Q_i\}_{i=1}^q,d,M,g,K,A$ be as in Theorem \ref{1.1}. Assume that $g$ is ramified over $Q_j$ with multiplicity at least $m_j$ on $A$ for each $j$ (i.e., $\nu_{Q_j(f)}\ge m_j$ on $\supp\nu_{Q_j(f)}\cap A$) and
$$\sum_{j=1}^q(1-\dfrac{M}{m_j})>\dfrac{(2N-k+1)(M+1)}{k+1}+\dfrac{(2N-k+1)M(M+1)}{2d(k+1)}.$$
Then $S$ has finite total curvature.

In particular, if $g$ intersects only a finite number of times the hypersurfaces $Q_1,\ldots,Q_q$ and 
$$q>\dfrac{(2N-k+1)(M+1)}{k+1}+\dfrac{(2N-k+1)M(M+1)}{2d(k+1)}$$
then $S$ must have finite total curvature.
\end{corollary}
Then, this corollary immediately implies Theorem C mentioned above of X. Mo.
\begin{proof}
Indeed, since $g$ is ramified over $Q_j$ with multiplicity at least $m_j$ on $A$ for each $j$, we have
$$ \delta^{H,M}_{g,S}\ge 1-\dfrac{M}{m_j}\ (1\le j\le q).$$
Then
$$\sum_{j=1}^q\delta^{H,M}_{g,A}(Q_j)>\dfrac{(2N-k+1)(M+1)}{k+1}+\dfrac{(2N-k+1)M(M+1)}{2d(k+1)}.$$
By Corollary \ref{1.2}, $S$ has finite total curvature.
\end{proof}
Now, combining Corollary \ref{1.2} and Theorem \ref{1.3}, we may prove the following.
\begin{corollary}\label{1.6}
Let $S, V,\{Q_i\}_{i=1}^q,d,M,g$ be as in Theorem \ref{1.1}. Then we have
$$\sum_{j=1}^q\delta^{H,M}_{g,S}(Q_j)\le \dfrac{(2N-k+1)(M+1)}{k+1}+\dfrac{(2N-k+1)M(M+1)}{2d(k+1)}.$$
In particular, the image of $S$ under $g$ cannot intersects only a finite number of times with hyperplanes $Q_1,\ldots,Q_q$ if $q>\dfrac{(2N-k+1)(M+1)(M+2d)}{2d(k+1)}$.
\end{corollary}
\begin{proof}
Suppose contrarily that
$$\sum_{j=1}^q\delta^{H,M}_{g,S}(Q_j)> \dfrac{(2N-k+1)(M+1)}{k+1}+\dfrac{(2N-k+1)M(M+1)}{2d(k+1)}.$$
By Theorem \ref{1.2}, this yields that $S$ has finite total curvature. Applying Theorem \ref{1.3}, we have
$$\sum_{j=1}^q\delta^{S,M}_{g,S}(Q_j)\le \dfrac{(2N-k+1)(M+1)}{k+1}+\dfrac{(2N-k+1)M(M+1)}{2d(k+1)}.$$
Since $\delta^{H,M}_{g,S}(Q_j)\le \delta^{S,M}_{g,S}(Q_j)$, we obtain
$$\sum_{j=1}^q\delta^{H,M}_{g,S}(Q_j)\le \dfrac{(2N-k+1)(M+1)}{k+1}+\dfrac{(2N-k+1)M(M+1)}{2d(k+1)}.$$
This contradicts the supposition.
The desired inequality of the corollary must hold.
\end{proof}

As we known, the first results on the value distribution of the Gauss map with hypersurfaces are given by S. D.Quang, D. D. Thai, P. D. Thoan in \cite{QTT}. However, they only consider the case of algebraic complete minimal surfaces. In their result, they assumed that the Gauss map is algebraic nondegenerate. Actually, from their proof, they only need the condition that the Gauss map is nondegenerate over $I_d(V)$. For the convenience, we state here that result.

\noindent
{\bf Theorem F} (S. D. Quang \textit{et all.} \cite[Theorem 1.2]{QTT}, with corrected statement). {\it Let $x:S\rightarrow\R^m$ be a non-flat complete regular minimal surface with finite total curvature and let $V$ be a projective subvariety of $\P^{m-1}(\C)$ of dimension $k$. Let $G:S\rightarrow\P^{m-1}(\C)$ be its generalized Gauss map. Let $\{Q_i\}_{i=1}^q$ be hypersurfaces of $\mathbb P^{m-1}(\mathbb C)$ in $N$-subgeneral position with respect to $V$ with $\deg Q_i=d_i$. Let $d=lcm (d_1,\ldots ,d_q)$. Assume that  $G$ is nondegenerate over $I_d(V)$ and is ramified over hypersurfaces $Q_j$ with multiplicity at least  $m_j$ for each $j$. Then, we have
$$ \sum_{i=1}^q\left (1-\dfrac{H_V(d)-1}{m_j}\right )< \frac{(2N-k+1)H_V(d)(H_V(d)+2d-1)}{2d(k+1)}.$$}

We now give an example to show the existance of the Gauss map fulfilled the above assumption of Theorem F with $\dim V\ge 2$. 
\begin{example}{\rm
Let $S=\C=\P^1(\C)\setminus{\infty}$. Consider the following monomials on $S$:
$$ f_j=z^{n_j},g_j=\sqrt{-1}z^{n_j}\ \forall 1\le j\le n\ (n\ge 3),$$
where $n_1=0,\ldots,n_k=1+2(n_{1}+\cdots+n_{k-1})$ for all $k=2,\ldots,n$. We see that 
$$f_1^2+g_1^2+\cdots+f_n^2+g_n^2=0.$$
By Theorem 1.2.5 in \cite{Fu93}, there exists a minimal surface $x=(x_1,\ldots,x_{2n}):S\rightarrow R^{2n}$ with the Gauss map is given by
$$ G=(f_1:g_1:f_2:g_2:\cdots:f_n:g_n):S\rightarrow\P^{2n-1}(\C),$$
which extends holomorphically over $\P^1(\C)$, and the induced metric is given by
$$ ds^2=2(|\omega_1|^2+\cdots+|\omega_{2n}|^2), $$
where $\omega_{2k-1}=f_kdz,\omega_{2k}=g_kdz.$ We denote by $(x_1:y_1:x_2:y_2:\cdots:x_n:y_n)$ the homogeneous coordinates on $\P^{2n-1}(\C)$. Let $Q$ be an arbitrary hypersurface of degree 2 with the defining homogeneous polynomials (denoted again by $Q$) of the form
$$ Q=\sum_{1\le i\le j\le n}(a_{ij}x_ix_j+b_{ij}x_iy_j+c_{ij}y_iy_j).$$
Then
$$ Q(G)=\sum_{1\le i\le j\le n}(a_{ij}+\sqrt{-1}b_{ij}-c_{ij})z^{n_i+n_j}.$$
It is easy to see that $\{z^{n_i+n_j};1\le i\le j\le n\}$ is linearly independent. Therefore, $Q(G)\equiv 0$ if and only if $a_{ij}+\sqrt{-1}b_{ij}-c_{ij}=0$ for all $1\le i\le j\le n$.
Denote by $V$ the intersection of all hypersurfaces $Q$ of degree 2 such that the image of $G$ is contained in $Q$. Then $G:S\rightarrow V$ is nondegenerate over $I_2(V)$. It is clear that $V$ contains all point $(z_1:\sqrt{-1}z_1:z_2:\sqrt{-1}z_2:\cdots:z_n:\sqrt{-1}z_n)$ for every $z_1,\ldots,z_n$ not all zero. Hence, $\dim V\ge n-1$.}
\end{example}

Here, Corollary \ref{1.4} generalizes Theorem F to the case of complete minimal surfaces which may not be algebraic. Our results in this paper are the initial results for the general case of complete minimal surfaces, which generalize for previous results in this area, for instance see \cite{Ha18a,Ha18b,HTP,Fu83,Fu90,M94,RJ}. The results and method of this paper have been applied by C. Lu and X. Chen to study the unicity of Gauss maps sharing hypersurfaces in projective spaces \cite{LC}.

\section{Defect relation of holomorphic curves from a punctured disk into a projective variety with a family of hypersurfaces}

\noindent
{\bf A. Nevanlinna functions.}\ For each $s>0$, we define the punctured disk
$$\Delta_{s,\infty}=\{z\in\C\ |\ s\le |z|<\infty\}.$$
In this section, we always assume that functions and mappings from $\Delta_{s,\infty}$ are defined on an open neighborhood of $\Delta_{s,\infty}$ in $\C$.

For a divisor $\nu$ on an open neighborhood of $\Delta_{s,\infty}$ and for a positive integer $p$ or $p= \infty$, we define the counting function of $\nu$ by
\begin{align*}
\nu^{[p]}(z)&=\min\ \{p,\nu(z)\},\\
n(t,s)& =\sum\limits_{s\le |z|\leq t} \nu (z).
\end{align*}
Similarly, we define $n^{[p]}(t,s).$
The counting function of $\nu$ is defined by
$$ N(r,s,\nu)=\int\limits_{s}^r \dfrac {n(t,s)}{t}dt \quad (s<r<\infty).$$
Similarly, define  \ $N(r,s,\nu^{[p]})$ and denote it by \ $N^{[p]}(r,s,\nu)$.

Let $\varphi : \Delta_{s,\infty}\longrightarrow \C$ be a meromorphic function. Denote by $\nu^0_\varphi$ the zero divisor of $\varphi$. Define
$$N_{\varphi}(r,s)=N(r,s,\nu^0_{\varphi}), \ N_{\varphi}^{[p]}(r,s)=N^{[p]}(r,s,\nu^0_{\varphi}).$$
For brevity, we will omit the character $^{[p]}$ if $p=\infty$. The proximity function of $\varphi$ (with respect to the point $\infty$) is defined by
$$ m(r,s,\varphi)=\int\limits_{0}^{2\pi}\log^+|\varphi(re^{i\theta})|\dfrac{d\theta}{2\pi}-\int\limits_{0}^{2\pi}\log^+|\varphi(se^{i\theta})|\dfrac{d\theta}{2\pi}\ (s<r<\infty),$$
where $\log^+x=\log(\max\{0,x\})$ for a real number $x$. The Nevanlinna's characteristic function of $\varphi$ is defined by
$$ T(r,s,\varphi):=m(r,s,\varphi)+N_{1/\varphi}(r,s).$$

Throughout this paper, we fix a homogeneous coordinates system $(x_0:\cdots :x_n)$ on $\P^n(\C)$. Let $f : \Delta_{0,\infty} \longrightarrow \P^n(\C)$ be a holomorphic curve with a reduced representation $F = (f_0, \ldots,f_n)$ on $\Delta_{s,\infty}$, which means $f_0,\ldots,f_n$ are holomorphic functions on $\Delta_{s,\infty}$ without common zeros. We set 
$$\|F\|=(|f_0|^2+\cdots+|f_n|^2)^{1/2}.$$
The characteristic function of $f$ is defined by 
$$ T_f(r,s)=\int_{0}^{2\pi}\log\|F(re^{i\theta})\|\dfrac{d\theta}{2\pi}-\int_{0}^{2\pi}\log\|F(se^{i\theta})\|\dfrac{d\theta}{2\pi}\ (s<r<R). $$
For a meromorphic function $\varphi$, if we consider $\varphi$ as a holomorphic curve into $\P^1(\C)$, then
$$ T_\varphi(r,s)=T(r,s,\varphi)+O(1).$$

\vskip0.2cm
\noindent
{\bf B. Auxiliary results.}\ We need the following.

\begin{proposition}[{cf. \cite[Proposition 4.5]{Fu85}}]\label{pro2.1}
Let $f_0,\ldots ,f_n$ be holomorphic functions on $\Delta_{s,\infty}$ such that $\{f_0,\ldots ,f_n\}$ are  linearly independent over $\C.$ Then the Wronskian 
$$ W(f_0,\ldots,f_n)=\det\left(f_j^{(i)}\right)_{0\le i,j\le n},$$ 
where $f_j^{(i)}$ is the $i^{th}$-derivative of $f_j$, satisfying:

(i)\  $W(f_0,\ldots,f_n)\not\equiv 0.$ 

(ii) $W(hf_0,\ldots,hf_n)=h^{n+1}W(f_0,\ldots,f_n)$ for any meromorphic function $h$ on $\Delta_{s,\infty}.$
\end{proposition}

\begin{lemma}[{cf. \cite[Lemma 2.2]{No81}}]\label{lem2.2}
Let $f:\Delta_{s,\infty}\rightarrow\P^n(\C)$ be
 a holomorphic curve. We have 
$$ \biggl\| \quad m\left(r,s,\dfrac{f^{(k)}}{f}\right)\le O(\log T_f(r,s))+O(\log r).$$
\end{lemma}

\begin{lemma}[{cf. \cite[Lemma 3]{QA}}]\label{lem2.3}
Let $V$ be a complex projective subvariety of $\P^n(\C)$ of dimension $k\ (k\le n)$. Let $Q_1,...,Q_q$ be $q\ (q>2N-k+1)$ hypersurfaces in $\P^n(\C)$ in $N$-subgeneral position with respect to $V$ of the common degree $d.$ Then, there are positive rational constants $\omega_i\ (1\le i\le q)$ satisfying the following:

i) $0<\omega_i \le 1\  \forall i\in\{1,...,q\}$,

ii) Setting $\tilde \omega =\max_{j\in Q}\omega_j$, one gets
$$\sum_{j=1}^{q}\omega_j=\tilde \omega (q-2N+k-1)+k+1.$$

iii) $\dfrac{k+1}{2N-k+1}\le \tilde\omega\le\dfrac{k}{N}.$

iv) For each $R\subset \{1,...,q\}$ with $\sharp R = N+1$, then $\sum_{i\in R}\omega_i\le k+1$.

v) Let $E_i\ge 1\ (1\le i \le q)$ be arbitrarily given numbers. For each $R\subset \{1,...,q\}$ with $\sharp R = N+1$,  there is a subset $R^o\subset R$ such that $\sharp R^o=\rank \{Q_i\}_{i\in R^o}=k+1$ and 
$$\prod_{i\in R}E_i^{\omega_i}\le\prod_{i\in R^o}E_i.$$
\end{lemma}

We know the following characterization of a removable singularity
 (cf.\ \cite{No81}).
\begin{lemma}\label{lem2.4}
Let $f$ be a holomorphic curve from $\Delta_{s,\infty}$ into a projective subvariety $V$ of $\P^n(\C)$. Then $f$ extends holomorphically   in a neighborhood of $\infty$ if and only if
 $$\liminf_{r \to \infty} T_f(r,s)/(\log r) < \infty.$$
 \end{lemma}

Let $\{Q_i\}_{i\in R}$ be a family of hypersurfaces in $\P^n(\C)$ of the common degree $d$. Assume that each $Q_i$ is defined by
$$ \sum_{I\in\mathcal I_d}a_{iI}x^I=0, $$
where $\mathcal I_d=\{(i_0,...,i_n)\in \mathbb N_0^{n+1}\ :\ i_0+\cdots + i_n=d\}$, $I=(i_0,...,i_n)\in\mathcal I_d,$ $x^I=x_0^{i_0}\cdots x_n^{i_n}$ and $(x_0:\cdots: x_n)$ is a homogeneous coordinate of $\P^n(\C)$.

Let $f:\Delta_{s,\infty}\rightarrow V\subset\P^n(\C)$ be a meromorphic mapping into $V$ with a reduced representation $F=(f_0,\ldots,f_n)$. We define
$$ Q_i(F)=\sum_{I\in\mathcal I_d}a_{iI}f^I ,$$
where $f^I=f_0^{i_0}\cdots f_n^{i_n}$ for $I=(i_0,...,i_n)$. Since the divisor $\nu_{Q(F)}$ does not depend on the choice of the reduced representation of $f$, we denote it by $\nu_{Q(f)}$. Then we can consider $f^*Q_i=\nu_{Q_i(f)}$ as divisors. 

\begin{lemma}[{cf. \cite[Lemma 4]{QA}}]\label{lem2.5}
Let $\{Q_i\}_{i\in R}$ be a set of hypersurfaces in $\P^n(\C)$ of the common degree $d$ and let $f$ be a meromorphic mapping of $\Delta_{s,\infty}$ into $\P^n(\C)$. Assume that $\bigcap_{i\in R}Q_i\cap V=\varnothing$. Then there exist positive constants $\alpha$ and $\beta$ such that
$$\alpha \|f\|^d \le  \max_{i\in R}|Q_i(f)|\le \beta \|f\|^d.$$
\end{lemma} 

\begin{lemma}[{cf. \cite[Lemma 5]{QA}}]\label{lem2.6}
Let $V$ be a projective subvariety of dimension $k$ of $\P^n(\C)$. Let $\{Q_i\}_{i=1}^q$ be a set of $q$ hypersurfaces in $\P^n(\C)$ of the common degree $d$ and let $M=H_d(V)-1$. Then there exist $M-k$ hypersurfaces $\{T_i\}_{i=1}^{M-k}$ in $\P^n(\C)$ such that for any subset $R\subset\{1,...,q\}$ with $\sharp R=\rank \{[Q_i]\}_{i\in R}=k+1,$ we get 
$$\rank \{\{[Q_i]\}_{i\in R}\cup\{[T_i]\}_{i=1}^{M-k}\}=M+1.$$
\end{lemma}

\vskip0.2cm
\noindent
{\bf C. Defect relation for holomorphic curves}
\begin{theorem}\label{thm2.7} 
Let $V$ be a complex projective subvariety of $\P^n(\C)$ of dimension $k\ (k\le n)$. Let $\{Q_i\}_{i=1}^q$ be hypersurfaces of $\P^n(\C)$ in $N$-subgeneral position with respect to $V$ with $\deg Q_i=d_i\ (1\le i\le q)$. Let $d$ be the least common multiple of $d_i$'s, i.e., $d=lcm (d_1,...,d_q)$, and let $M=H_d(V)-1$. Let $f$ be a holomorphic mapping of $\Delta_{s,\infty}$ into $V$ such that $f$ is nondegenerate over $I_d(V)$. Then, we have
$$ \biggl |\biggl |\ \left (q-\dfrac{(2N-k+1)(M+1)}{k+1}\right )T_f(r,s)\le \sum_{i=1}^{q}\dfrac{1}{d_i}N^{[M]}_{Q_i(f)}(r,s)+o(T_f(r,s)).$$
\end{theorem}
\begin{proof}
By replacing $Q_i$ with $Q_i^{d/\deg Q_i} \ (i=1,...,q)$ if necessary, we may assume that all $Q_i\ (1\le i\le q)$ have the same degree $d$.
It is easy to see that there is a positive constant $C_1$ such that $\|f\|^d\ge C_1|Q_i(f)|$ for every $1\le i\le q.$
Set $ Q:=\{1,\cdots ,q\}$. Let $\{\omega_i\}_{i=1}^q$ and $\tilde\omega$ be as in Lemma \ref{lem2.3} for the family $\{Q_i\}_{i=1}^q$.  Let $\{T_i\}_{i=1}^{M-k}$ be $(M-k)$ hypersurfaces in $\P^n(\C)$, which satisfy Lemma \ref{lem2.6}. 

Take a $\C$-basis $\{[A_i]\}_{i=0}^{M}$ of $I_d(V)$, where $A_i\in H_d$. Let $F=(f_0,\ldots,f_n)$ be a reduced representation of $f$. Since $f$ is nondegenerate over $I_d(V)$, $\{A_i(F); 0\le i\le M\}$ is linearly independent over $\C$. Then
$$W\equiv\det\bigl (A_j(F)^{(i)}\bigl )_{0\le i,j\le M}\not\equiv 0.$$

For each $R^o=\{i_0,...,i_k\}\subset\{1,...,q\}$ with $\rank \{Q_i\}_{i\in R^o}=\sharp R^o=k+1$, set 
$$W_{R^o}\equiv\det\bigl (A_{i_j}(F)^{(l)}(0\le j\le k),T_j(F)^{(l)} (1\le j\le M)\bigl )_{0\le l\le M}.$$
Since $\rank \{Q_{i_j} (0\le j\le k),T_j (1\le j\le M-k)\}=M+1$, there exists a nonzero constant  $C_{R^o}$ such that $W_{R^o}=C_{R^o}\cdot W$. 

We denote by $\mathcal R^o$ the family of all subsets $R^o$ of $\{1,...,q\}$ satisfying 
$$\rank \{Q_i\}_{i\in R^o}=\sharp R^o=k+1.$$

Let $z$ be a fixed point. For each $R\subset Q$ with  $\sharp R=N+1,$  we choose $R^{o}\subset R$ such that $R^o\in\mathcal R^o$ and $R^o$ satisfies Lemma \ref{lem2.3} v) with respect to numbers $\bigl \{\dfrac{C_1 \|F(z)\|^d}{|Q_i(F)(z)|}\bigl \}_{i=1}^q$.  On the other hand, there exists $\bar R\subset Q$ with $\sharp \bar R=N+1$ such that $|Q_{i}(F)(z)|\le |Q_j(F)(z)|,\forall i\in \bar R,j\not\in \bar R$. Since $\bigcap_{i\in \bar R}Q_i=\varnothing$, by Lemma \ref{lem2.5}, there exists a positive constant $C_2$ (chosen common for all $R$) such that
$$ C_2\|F\|^d(z)\le \max_{i\in \bar R}|Q_i(F)(z)|. $$
Then, we get
\begin{align*}
\dfrac{\|F(z)\|^{d(\sum_{i=1}^q\omega_i)}|W(z)|}{|Q_1(F)(z)|^{\omega_1}\cdots |Q_q(F)(z)|^{\omega_q}}
&\le\dfrac{|W(z)|}{C_2^{q-N-1}C_1^{N+1}}\prod_{i\in \bar R}\left (\dfrac{C_1\|F(z)\|^d}{|Q_i(F)(z)|}\right )^{\omega_i}\\
&\le C_3\dfrac{|W(z)|\cdot \|F\|^{d(k+1)}(z)}{\prod_{i\in \bar R^o}|Q_i(F)|(z)}\\
&\le C_4\dfrac{|W_{\bar R^o}(z)|\cdot \|F\|^{d(M+1)}(z)}{\prod_{i\in \bar R^o}|Q_i(F)|(z)\prod_{i=1}^{M-k}|T_i(F)|(z)},
\end{align*}
where $C_3, C_4$ are positive constants. 

Put $S_{\bar R}=C_4\dfrac{|W_{\bar R^o}|}{\prod_{i\in \bar R^o}|Q_i(F)|\prod_{i=1}^{M-k}|T_i(F)|}$. By the Lemma on logarithmic derivative, it is easy to see that 
$$\|\ \int\limits_0^{2\pi}\log^{+}S_{\bar R}(re^{i\theta})\dfrac{d\theta}{2\pi}=O(\log^+T_f(r,s))+O(\log r).$$

Therefore, for each  $z\in \Delta_{s,\infty}$, we have
\begin{align*}
\log \left (\dfrac{\|F(z)\|^{d(\sum_{i=1}^q\omega_i)}|W(z)|}{|Q_1(F)(z)|^{\omega_1}\cdots |Q_q(F)(z)|^{\omega_q}}\right )\le \log \left (\|F(z)\|^{d(M+1)}\right )+\sum_{R\subset Q,\sharp R=N+1}\log^+S_R.
\end{align*}
Since $\sum_{i=1}^q\omega_i=\tilde\omega(q-2N+k-1)+k+1$ and by integrating both sides of the above inequality over $S(r),$  we have
\begin{align}\label{2.8}
\begin{split}
\|\   d\biggl(q-2N+k-1&-\dfrac{M-k}{\tilde\omega}\biggl)T_f(r,s)\le\sum_{i=1}^{q}\dfrac{\omega_i}{\tilde\omega}N_{Q_i(f)}(r,s)\\
&-\dfrac{1}{\tilde\omega}N_{W}(r,s)+O(\log^+T_f(r,s))+O(\log r).
\end{split}
\end{align}
By usual arguments (see \cite{QA}), we have
$$\sum_{i=1}^q\omega_iN_{Q_i(f)}(r,s)-N_{W}(r,s)\le \sum_{i=1}^q\omega_iN^{[M]}_{Q_i(f)}(r,s).$$
Combining this and (\ref{2.8}), we obtain
\begin{align*}
\|\  & d\left(q-2N+k-1-\dfrac{M-k}{\tilde\omega}\right)T_f(r,s)\\
&\le\sum_{i=1}^{q}\dfrac{\omega_i}{\tilde\omega}N^{[M]}_{Q_i(f)}(r,s)+O(\log^+T_f(r,s))+O(\log r)\\
&\le \sum_{i=1}^{q}N^{[M]}_{Q_i(f)}(r,s)+O(\log^+T_f(r,s))+O(\log r).
\end{align*}
Since $\tilde\omega \ge \dfrac{k+1}{2N-k+1}$, the above inequality  implies that
$$ \biggl\|\quad  d\left (q-\dfrac{(2N-k+1)(M+1)}{k+1}\right )T_f(r,s)\le \sum_{i=1}^{q}N^{[M]}_{Q_i(f)}(r,s)+O(\log^+T_f(r,s))+O(\log r).$$
Hence, the theorem is proved.
\end{proof}

For each hypersurface $Q$ of degree $d$, we define the truncated defect of $f$ with respect to $Q$ by
$$ \delta^{[m]}_f(Q)=1-\lim\underset{r\rightarrow\infty}{\rm sup}\dfrac{N_{Q(f)}^{[m]}(r,s)}{dT_f(r,s)}.$$
As we known that
$$\delta^{H,M}_{f,\Delta_{s,\infty}}(Q)\le\delta_{f,\Delta_{s,\infty}}^{S,M}(Q)\le\delta_f^{[M]}(Q)\le 1.$$

From Lemma \ref{lem2.4} and Theorem \ref{thm2.7}, we have the following corollary.
\begin{corollary}\label{cor2.9}
Let $f$ be a holomorphic map from $\Delta_{s,\infty}$ into a projective subvariety $V$ of $\P^n(\C)$ of dimension $k$. Let $Q_1,\ldots,Q_q$ be distinct $q$ hypersurface in $N$-subgeneral position with respect to $V$. Assume that $f$ has an essential singularity at $\infty$, then
$$\sum_{j=1}^q\delta^{H,M}_{f,\Delta_{s,\infty}}(Q_j)\le\sum_{j=1}^q\delta^{[M]}_f(Q_j)\le \dfrac{(2N-k+1)H_V(d)}{k+1}.$$
\end{corollary}

\section{Sum to Product inequality for holomorphic curves and families of hypersurfaces}

Let $S$ be an open Riemann surface and let $ds^2$ be a pseudo-metric on $S$ which is locally written as $ds^2=\lambda^2|dz|^2$, where $\lambda$ is a nonnegative real-value function with mild singularities and $z$ is a holomorphic local coordinate. The divisor of $ds^2$ is defined by $\nu_{ds}:=\nu_\lambda$ for each local expression $ds^2=\lambda^2|dz|^2$, which is globally well-defined on $S$. We say that $ds^2$ is a continuous pseudo-metric if $\nu_{ds}\ge 0$ everywhere.

The Ricci of  $ds^2$ is defined by 
$$ \mathrm{Ric}_{ds^2}=dd^c\log\lambda^2$$
for each local expression  $ds^2=\lambda^2|dz|^2$. This definition is globally well-defined on $S$.

The continuous pseudo-metric $ds^2$ is said to have strictly negative curvature on $S$ if there is a positive constant $C$ such that
$$ -\mathrm{Ric}_{ds^2}\ge C\Omega_{ds^2}, $$
where $\Omega_{ds^2}=\lambda^2\cdot\dfrac{\sqrt{-1}}{2}\cdot dz\wedge d\bar z$.

Let $f$ be a holomorphic map of $S$ into a projective subvariety $V$ of dimension $k$ of $\P^n(\C)$. Let $d$ be a positive integer and assume that $f$ is nondegenerate over $I_d(V)$. We fix a $\C$-ordered basis $\mathcal V=([v_0],\ldots,[v_M])$ of $I_d(V)$, where $v_i\in H_d$ and $M=H_V(d)-1$.  

Let $F=(f_0,\ldots,f_n)$ be a local reduced representation of $f$ (in a local chat $(U,z)$ of $S$). 
Consider the holomorphic map
$$ F_{\mathcal V}=(v_0(F),\ldots,v_M(F))$$
and 
$$F_{\mathcal V,p} = (F_{\mathcal V,p})_z := F_{\mathcal V}^{(0)} \wedge F_{\mathcal V}^{(1)} \wedge \cdots \wedge F_{\mathcal V}^{(p)} : S\rightarrow \bigwedge_{p+1}\C^{M+1}$$
for $0\le p\le M$, where 
\begin{itemize}
\item $F_{\mathcal V}^{(0)}:=F_{\mathcal V}=(v_0(F),\ldots,v_M(F))$,
\item $F_{\mathcal V}^{(l)}=(F_{\mathcal V}^{(l)})_z:=\left (v_0(F)^{(l)}_z,\ldots, v_M(F)^{(l)}_z\right)$ for each $l=0, 1,\ldots , p$,
\item $v_i(F)^{(l)}_z \ (i =0,\ldots, M)$ is the $l^{th}$- derivatives of $v_i(F)$ taken with respect to $z$.
\end{itemize}
The norm of $F_{\mathcal V,p}$ is given by
$$|F_{\mathcal V,p}|:=\left (\sum_{0\le i_0<i_1<\cdots<i_p\le M}\left |W_z(v_{i_0}(F),\ldots,v_{i_p}(F))\right|^2\right)^{1/2}, $$
where 
$$W_z(v_{i_0}(F),\ldots,v_{i_p}(F)):=\det\left (v_{i_j}(F)^{(l)}_z\right)_{0\le l,j\le p}$$ 
We have some fundamental properties of the wronskian of holomorphic function $h_0,h_1,\ldots,h_p$ as follows:
\begin{itemize}
\item $W_{\xi}(h_0,\ldots,h_p)=W_z(h_0,\ldots,h_p) (\xi'_z)^{\frac{p(p+1)}{2}}$,
\item $W_z(hh_0,\ldots,hh_p)=h^{p+1}W_z(h_0,\ldots,h_p)$,
\item $h_0,\ldots,h_p$ are $\C$-linearly dependent if and only if $W_z(h_0,\ldots,h_p)\equiv 0$.
\end{itemize}

We use the same notation $\langle,\rangle$ for the canonical hermitian product on $\bigwedge^{l+1}\C^{M+1}\ (0\le l\le M)$. For two vectors $A\in \bigwedge^{k+1}\C^{M+1}\ (0\le k\le M)$ and $B\in\bigwedge^{p+1}\C^{M+1}\ (0\le p\le k)$, there is one and only one vector $C\in\bigwedge^{k-p}\C^{M+1}$ satisfying 
$$ \langle C,D\rangle=\langle A,B\wedge D\rangle\ \forall D\in \bigwedge^{k-p}\C^{M+1}.$$
The vector $C$ is called the interior product of $A$ and $B$, and defined by $A\vee B$.

Now, for a hypersurface $Q$ of degree $d$ in $\P^n(\C)$ defined by a homogeneous polynomial in $H_d$ denoted again by $Q$. Then we have
$$[Q]=\sum_{i=0}^Ma_i[v_i].$$
Hence, we associate $Q$ with the vector $(a_0,\ldots,a_M)\in\C^{M+1}$ and define $F_{\mathcal V,p}(Q)=F_{\mathcal V,p}\vee H$. Then, we may see that
\begin{align*}
F_{\mathcal V,0}(Q)&=a_0v_0(F)+\cdots+a_Mv_M(F)=Q(F),\\ 
|F_{\mathcal V,p}(Q)|&=\left (\sum_{0\le i_1<\cdots<i_p\le M}\sum_{\ell\ne i_1,\ldots,i_p}a_l\left |W_z(v_\ell(F),v_{i_1}(F),\ldots,v_{i_p}(F))\right|^2\right)^{1/2}
\end{align*}
Finally, for $0\le p\le M$, the $p^{th}$-contact function of $f$ for $Q$ with respect to $\mathcal V$ is defined (not depend on the choice of the local coordinate) by
$$\varphi_{\mathcal V,p}(Q):=\dfrac{|F_{\mathcal V,p}(Q)|^2}{|F_{\mathcal V,p}|^2}.$$

For each $p\ (0\le p\le M-1)$, let $M_p=\binom{M+1}{p+1}-1$ and $\pi_p$ be the canonial projection from $\bigwedge^{p+1}\C^{M+1}\sim\C^{M_p+1}$ onto $\P^{M_p}(\C)$. Denote by $\Omega_p$ the pullback of the Fubini-Study form on $\P^{M_p}(\C)$ by the map $\pi\circ F_{\mathcal V,p}$, i.e., $\Omega_p = dd^c\log |F_{\mathcal V,p}|^2$. 
\begin{proposition}[{cf. \cite[Proposition 2.5.1]{Fu93}}]\label{pro3.1}
Let $S,V,d,\mathcal V$ and $M$ be as above. For each positive $\epsilon$ there exists a constant $\delta_0(\epsilon)$, depending only on $\epsilon$, such that for any hypersurface $Q$ of degree $d$ in $\P^n(\C)$ and any constant $\delta>\delta_0(\epsilon)$
$$ dd^c\log\dfrac{1}{\log(\delta/\varphi_{\mathcal V,p}(Q))}\ge\dfrac{\varphi_{\mathcal V,p}(Q)}{\varphi_{\mathcal V,p+1}(Q)\log(\delta/\varphi_{\mathcal V,p}(Q))}\Omega_p-\epsilon\Omega_p.$$
\end{proposition}
We note that, the original proposition of H. Fujimoto (cf. \cite[Proposition 2.5.1]{Fu93}) states for hyperplanes. However, by using Serge embedding $\P^n(\C)$ into $\P^M(\C)$, we automatically deduce the above proposition.

\begin{theorem}\label{thm3.2}
Let $S,V,d,\mathcal V$ and $M$ be as above. Let $f:S\rightarrow V\subset\P^n(\C)$ be a holomorphic curve and let $Q_1,\ldots, Q_q$ be hypersurfaces in $\P^n(\C)$ located in $N-$subgeneral position with respect to $V$ of the same degree $d$, where $q>\frac{(2N-k+1)(M+1)}{k+1}$. Assume that $f$ is nondegenerate over $I_d(V)$ and have a local reduced representation $F=(f_0,\ldots,f_n)$ on a local holomorphic chat $(U,z)$. Let $\omega_j\ (1\le j\le q)$ be the Nochka weights for these hypersurfaces (defined in Lemma \ref{lem2.3}). For an arbitrarily given $\delta >1$ and $0\le p\le M-1$, we set
$$\Phi_{\mathcal V,jp}:=\dfrac{\varphi_{\mathcal V,p+1}(Q_j)}{\varphi_{\mathcal V,p}(Q_j)\log^2\left (\delta/\varphi_{\mathcal V,p}(Q_j)\right)}.$$
Then, there exists a positive constant $C_{\mathcal V,p}$ depending only in $\mathcal V,p$ and $Q_j\ (1\le j\le q)$ such that
$$ \sum_{j=1}^q\omega_j\Phi_{\mathcal V,jp}\ge C_{\mathcal V,p}\left (\prod_{j=1}^q\Phi_{\mathcal V,jp}^{\omega_j}\right)^{1/(M-p)} $$
holds on $S-\bigcup_{1\le j\le q}\{z;\varphi_{\mathcal V,p}(Q_j)(z)=0\}$.
\end{theorem}
\begin{proof}
For each $j\ (1\le j\le q)$, we write
$$ [Q_j]=\sum_{i=0}^Ma_{ji}[v_i] $$
and set $a_j=(a_{j0},\ldots,a_{jM})\in\C^{M+1}$. We consider the set $\mathcal R_p$ of all subsets $R$ of $Q=\{1,2,\ldots, q\}$  such that $\rank_{\C}\{a_j\ |\ j\in R\}\le M-p$. For each $P\in G(M,p)$ (the Grassmanian manifold of all $(p+1)$-dimension sublinear spaces of $\C^{M+1}$), we take a decomposable $(p+1)$-vector $E$ such that
$$ P=\{X\in\C^{M+1}; E\wedge X=0\}$$
and set
$$ \psi_p(P)=\underset{R\in\mathcal R_p}{\max}\min\left\{\dfrac{\left|E\vee a_j\right|^2}{|E|^2};j\not\in R\right\}. $$
Then $\phi_p(P)$ depends only on $P$ and may be regarded as a function on the Grassmann manifold $G(M,p)$. For each nonzero $(k+1)$-vector $E=E_0\wedge E_1\wedge\ldots\wedge E_k$ we set
$$R=\{j\in Q; E\vee a_j=0\}.$$
As we known, $E\vee a_j=0$ means that $a_j$ is contained in the orthogonal complement of the vector space $\mathrm{Span}(E_0,\ldots, E_k)$. Then we see
$$\rank_{\C}\{a_j\ |\ j\in R\}=\dim\mathrm{Span}(a_j; j\in R)\le M - p,$$
namely, $R\in\mathcal R_p$. This yields that $\psi_p$ is positive everywhere on $G(M,p)$.
Since $\psi_p$ is obviously continuous and $G(M,p)$ is compact, we can take a positive constant $\delta$ such that $\psi_p(P)>\delta$ for each $P\in G(M,p)$.

Take a point $z$ with $F_{\mathcal V,p}(z)=0$. The vector space generated
by $F^{(0)}_{\mathcal V}(z), F^{(1)}_{\mathcal V}(z)$, ..., $F^{(p)}_{\mathcal V}(z)$ determines a point in $G(M,p)$. Therefore, there is a set $R$ in $\mathcal R_p$ with 
$$ \rank_{\C}\{a_j\ |\ j\in R\}\le M-p $$ 
such that $\varphi_{\mathcal V,p}(Q_j)(z)\ge\delta$ for all $j\not\in R$. Then, we can choose a finite positive constant $K$ depending only on $\mathcal V,Q_j\ (1\le j\le q)$ such that $\Phi_{\mathcal V,jp}(z)\le K$ for all $j\not\in R$. Set
$$T:= \{j; \Phi_{\mathcal V,jp}(z)> K\}, l:= \sum_{j\in T}\omega_j.$$

We consider the following two cases.

\textit{Case 1: }Assume that $T$ is an empty set. We have
\begin{align*}
\sum_{j=1}^q\omega_j\Phi_{\mathcal V,jp}&\ge\left (\sum_{j=1}^q\omega_j\right)\left (\prod_{j=1}^q\Phi_{\mathcal V,jp}^{\omega_j}\right)^{\frac{1}{\sum_j\omega_j}}\\ 
& \ge (p+1)K\left (\prod_{j=1}^q\left(\dfrac{\Phi_{\mathcal V,jp}}{K}\right)^{\frac{\omega_j}{M-p}}\right)^{\frac{M-p}{\sum_{j}\omega_j}}\\
&\ge (p+1)K\left (\prod_{j=1}^q\left(\frac{\Phi_{\mathcal V,jp}}{K}\right)^{\omega_j}\right)^{\frac{1}{M-p}},
\end{align*}
because $\sum_{j=1}^q\omega_j=\tilde\omega (q-2N+k-1)+k+1\ge \dfrac{q(k+1)}{2N-k+1}\ge M+1>M-p$.

\textit{Case 2: }Assume that $T\ne\emptyset$. We see that $T\subset R$ and so $\rank_{\C}\{a_j\ |\ j\in T\}\le M-p$ holds. We note that, there are at most $N$ indexes $j$ satisfying $\langle F^{(0)}_{\mathcal V}(z),a_j\rangle =0$ (i.e., $Q_j(F)(z)=0$) since the intersection of $V$ and arbitrary $N+1$ hyperplanes $Q_j\ (1\le j\le q)$ is empty. It yields that $\sharp T<N+1$, and hence $l\le \rank_{\C}\{a_j\ |\ j\in T\}\le M-p$. Then, we have
\begin{align*}
\sum_{j=1}^q\omega_j\Phi_{\mathcal V,jp}&\ge\sum_{j\in T}\omega_j\Phi_{\mathcal V,jp} \ge Kl\prod_{j\in T}\left(\frac{\Phi_{\mathcal V,jp}}{K}\right)^{\frac{\omega_j}{l}}\\
&\ge Kl\prod_{j\in T}\left(\frac{\Phi_{\mathcal V,jp}}{K}\right)^{\omega_j/(M-p)}\ge C\prod_{j\in S}\left(\frac{\Phi_{\mathcal V,jp}}{K}\right)^{\omega_j/(M-p)},
\end{align*}
for some positive constant $C>0$, which depends only on $Q_1,\ldots,Q_q$.

From the above two cases, we get the conclusion of the theorem.
\end{proof}

\begin{theorem}\label{thm3.3}
Let the notations and the assumption be as in Theorem \ref{thm3.2}. Moreover, let $\tilde\omega$ be the Nochka constant for these hypersurfaces (defined in the Lemma \ref{lem2.3}). Then, for every $\epsilon>0$, there exist a positive number $\delta\ (>1)$ and $C$, depending only on $\mathcal V,\epsilon$ and $Q_j$ such that
\begin{align*}
dd^c&\log\dfrac{\prod_{p=0}^{M-1}|F_{\mathcal V,p}|^{2\epsilon}}{\prod_{1\le j\le q,0\le p\le M-1}\log^{2\omega_j}\left(\delta/\varphi_{\mathcal V,p}(Q_j)\right)}\\
&\ge C\left (\dfrac{|F_{\mathcal V,0}|^{2\left (\tilde\omega(q-(2N-k+1))-M+k\right)}|F_{\mathcal V,M}|^2}{\prod_{j=1}^q(|F_{\mathcal V,0}(Q_j)|^2\prod_{p=0}^{M-1}\log^2(\delta/\varphi_{\mathcal V,p}(Q_j)))^{\omega_j}}\right)^{\frac{2}{M(M+1)}}dd^c|z|^2.
\end{align*}
\end{theorem}

\begin{proof}
We denote by $A$ the left hand side. Then we have
$$ A=\epsilon\sum_{p=0}^{M-1}\Omega_p+\sum_{j=1}^q \omega_j\sum_{p=0}^{M-1}dd^c\log\dfrac{1}{\log^{2}\left (\delta/\varphi_{\mathcal V,p}(Q_j)\right)}.$$
Choose a positive number $\delta_0(\epsilon/\ell)$ with the properties as in Proposition \ref{pro3.1}, where $\ell=\sum_{j=1}^q\omega_j$. For an arbitrarily fixed $\delta\ge \delta_0(\epsilon/\ell)$, we obtain
\begin{align*}
A&\ge\epsilon\sum_{p=0}^{M-1}\Omega_p+\sum_{j=1}^q\omega_j\sum_{p=0}^{M-1}\left (\dfrac{2\varphi_{\mathcal V,p+1}(Q_j)}{\varphi_{\mathcal V,p}(Q_j)\log^2(\delta/\varphi_{\mathcal V,p}(Q_j))}-\dfrac{\epsilon}{\ell}\right)\Omega_p\\ 
& =\sum_{p=0}^{M-1}2\left (\sum_{j=1}^q\omega_j\Phi_{\mathcal V,jp}\right).
\end{align*}
Then, by Theorem \ref{thm3.2} we have
$$ A\ge C_1\sum_{p=0}^{M-1}2\left (\prod_{j=1}^q\Phi_{\mathcal V,jp}^{\omega_j}\right)^{\frac{1}{M-p}}\Omega_p, $$
for some positive constant $C_1>0$. Let $\Omega_p=h_pdd^c|z|^2$, we have
\begin{align*}
A&\ge 2C_1\sum_{p=0}^{M-1}\left (h_p^{M-p}\prod_{j=1}^q\Phi_{\mathcal V,jp}^{\omega_j}\right)^{\frac{1}{M-p}}dd^c|z|^2\\ 
& \ge C'_1\sum_{p=0}^{M-1}(M-p)\left (h_p^{M-p}\prod_{j=1}^q\Phi_{\mathcal V,jp}^{\omega_j}\right)^{\frac{1}{M-p}}dd^c|z|^2\\ 
&\ge C_2\prod_{p=0}^{M-1}\left (h_p^{M-p}\prod_{j=1}^q\Phi_{\mathcal V,jp}^{\omega_j}\right)^{\frac{2}{M(M+1)}}dd^c|z|^2,
\end{align*}
for some positive constants $C'_1>0,C_2>0$. On the other hand, we note that
\begin{align*}
\prod_{p=0}^{M-1}\Phi_{\mathcal V,jp}&=\prod_{p=0}^{M-1}\dfrac{\varphi_{\mathcal V,p+1}(Q_j)}{\varphi_{\mathcal V,p}(Q_j)}\dfrac{1}{\log^2(\delta/\varphi_{\mathcal V,p}(Q_j))}\\ 
& =\dfrac{|F_{0}|^2}{|F_{0}(Q_j)|^2}\prod_{p=0}^{M-1}\dfrac{1}{\log^2(\delta/\varphi_{\mathcal V,p}(Q_j))}
\end{align*}
and
$$ \prod_{p=0}^{M-1}h_p^{M-p}=\prod_{p=0}^{M-1}\left (\dfrac{|F_{\mathcal V,p-1}|^2|F_{\mathcal V,p+1}|^2}{|F_{\mathcal V,p}|^4}\right)^{M-p}=\dfrac{|F_{\mathcal V,M}|^2}{|F_{\mathcal V,0}|^{2(M+1)}},$$
because $\varphi_{\mathcal V,0}(Q_j)=|F_{\mathcal V,0}(Q_j)|/|F_{\mathcal V,0}|,\varphi_{\mathcal V,M}(Q_j)=1$. Therefore, we get
$$ A\ge C\left (\dfrac{|F_{\mathcal V,0}|^{2(\ell-M-1)}|F_{\mathcal V,M}|^2}{\prod_{j=1}^q(|F_{\mathcal V,0}(Q_j)|^2\prod_{p=0}^{M-1}\log^2(\delta/\varphi_{\mathcal V,p}(Q_j)))^{\omega_j}}\right)^{\frac{2}{M(M+1)}}dd^c|z|^2.$$
Since $\ell-M-1=\tilde\omega (q-2N+k-1)-M+k$, we have the conclusion of the theorem.
\end{proof}

\begin{theorem}[{cf. \cite[Proposition 2.5.7]{Fu93}}]\label{thm3.4}
Set $\sigma_p=p(p+1)/2$ for $0\le p\le M+1$ and $\tau_m=\sum_{p=1}^m\sigma_m$. Then, we have
$$ dd^c\log(|F_{\mathcal V,0}|^2\cdots |F_{\mathcal V,M-1}|^2)\ge\dfrac{\tau_M}{\sigma_M}\left(\dfrac{|F_{\mathcal V,0}|^2\cdots |F_{\mathcal V,M}|^2}{|F_{\mathcal V,0}|^{2\sigma_{M+1}}}\right)^{1/\tau_M}dd^c|z|^2. $$
\end{theorem}

\begin{theorem}\label{thm3.5}
With the notations and the assumption in Theorem \ref{thm3.3}, we have
$$ \nu_{\phi}+\sum_{j=1}^q\omega_j\cdot\min\{\nu_{Q_j(f)},M\}\ge 0 $$
where $\phi=\dfrac{|F_{\mathcal V,M}|}{\prod_{j=1}^q|Q_j(F)|^{\omega_j}}$ and $\nu_\phi$ is the divisor of the function $\phi$.
\end{theorem}

\begin{proof}
Indeed, let $z$ be a zero of some $Q_i(F)$. Since $\{Q_i\}_{i=1}^q$ is in $N$-subgeneral position, $z$ is not zero of more than $N$ functions $Q_i(F)$. Without loss of generality, we may assume that $z$ is zero of $Q_i(F)$ for each $1\le i\le k\le N)$ and $z$ is not zero of $Q_i(F)$ for each $i>N$. Put $R=\{1,...,N+1\}.$ Choose $R^1\subset R$ such that 
$\sharp R^1=\rank\{Q_i\}_{i\in R^1}=k+1$ and $R^1$ satisfies Lemma \ref{lem2.3} v) with respect to numbers $\bigl \{e^{\max\{\nu_{Q_i(F)}(z)-M,0\}} \bigl \}_{i=1}^q.$ Then we have
\begin{align*}
\nu_{F_{\mathcal V,M}}(z)&\ge \sum_{i\in R^1}\max\{\nu_{Q_i(f)}(z)-M,0\}\\
&\ge\sum_{i\in R}\omega_i \max\{\nu_{Q_i(f)}(z)-M,0\}. 
\end{align*}
Hence 
\begin{align*}
\sum_{i=1}^q\omega_i\nu_{Q_i(f)}(z)-\nu_{F_{\mathcal V,M}}(z)& =\sum_{i\in R}\omega_i\nu_{Q_i(f)}(z)-\nu_{F_{\mathcal V,M}}(z)\\ 
&\le \sum_{i\in R}\omega_i\nu_{Q_i(f)}(z)-\sum_{i\in R}\omega_i\max\{\nu_{Q_i(f)}(z)-M,0\}\\
&=\sum_{i=1}^q\omega_i\min\{\nu_{Q_i(f)}(z),M\}.
\end{align*}
The theorem is proved.
\end{proof}

\begin{lemma}[{Generalized Schwarz's Lemma \cite{A38}}]\label{lem3.6}  
Let $v$ be a non-negative real-valued continuous subharmonic function on $\Delta (R)=\{z\in\C\ |\ \|z\|<R\}$. If $v$ satisfies the inequality $\Delta\log v\ge v^2$ in the sense of distribution, then
$$v(z) \le \dfrac{2R}{R^2-|z|^2}.$$
\end{lemma}

\begin{lemma}\label{lem3.7} 
Let $V,\mathcal V, d, M,\{Q_i\}_{i=1}^q$ and $\{\omega_i\}_{i=1}^q$ be as in Theorem \ref{thm3.2}. Let $f:\Delta (R)\rightarrow V\subset\P^n(\C)$ be a holomorphic map with a reduced representation $F=(f_0,\ldots,f_n)$, which is nondegenerate over $I_d(V)$. Assume that there are positive real numbers $\eta_j$, divisors $\nu_j$, continuous real-valued bounded functions $k_j$ with mild singularities on $S$, a positive constant $c$ such that $\nu_j(z)>c$ for each $z\in\supp (\nu_j)\ (1\le j\le q)$ and
$$ [\min\{M,\nu_{Q_j(f)}\}]+[\nu_j]=d\eta_j f^*\Omega+dd^c[\log|k_j|^2].$$
Let $h_j=|k_j|\cdot \|F\|^{d\eta_j}\ (1\le j\le q)$. Then for an arbitrarily given $\epsilon$ satisfying 
$$\gamma=\sum_{j=1}^q\omega_j (1-\eta_j)-M-1>\epsilon(\sigma_{M+1}+\sum_{j=1}^q\frac{\eta_j}{q}).$$ 
the pseudo-metric $d\tau^2=\eta^2|dz|^2$, where
$$ \eta=\left( \dfrac{|F_{\mathcal V,0}|^{\gamma-\epsilon(\sigma_{M+1}+\sum_{j=1}^q\frac{\eta_j}{q})}\prod_{j=1}^qh_{j}^{\omega_j+\frac{\epsilon}{q}}|F_{\mathcal V,M}|\prod_{p=0}^{M}|F_{\mathcal V,p}|^{\epsilon}}{\prod_{j=1}^q(|Q_j(F)|\prod_{p=0}^{M-1}\log (\delta/\varphi_{\mathcal V,p}(Q_j)))^{\omega_j}}\right )^{\frac{1}{\sigma_M+\epsilon\tau_M}}$$
and $\delta$ is the number satisfying the conclusion of Theorem \ref{thm3.3}, is continuous and has strictly negative curvature.
\end{lemma}
\begin{proof}
Without loss of generality, we may assume that $k_j\le 1$, and then $h_j\le \|F\|^{d\eta_j}\ (1\le j\le q)$. We see that the function $\eta$ is continuous at every point $z$ with $\prod_{j=1}^qQ_j(F)(z)\ne 0$. Now, for a point $z_0\in S$ such that $\prod_{j=1}^qQ_j(F)(z_0)= 0$, we have
\begin{align}\label{new1}
\nu_{h_j}(z_0)\ge \min\{M,\nu_{Q_j(F)}(z_0)\}.
\end{align}
Also, by Theorem \ref{thm3.5}, one has
\begin{align}\label{new4}
\nu_{F_{\mathcal V,M}}(z_0)-\sum_{j=1}^q\omega_j\nu_{Q_j(F)}(z_0)\ge  \sum_{j=1}^q\omega_j\min\{M,\nu_{Q_j(F)}(z_0)\}.
\end{align}
Combining the inequalities (\ref{new1}), (\ref{new4}) and the definition of $\eta$, we have
\begin{align*}
\nu_{\eta}(z_0)&\ge \frac{1}{\sigma_M+\epsilon\tau_M}\left (\nu_{F_{\mathcal V,M}}(z_0)+\sum_{j=1}^q(\omega_j+\frac{\epsilon}{q})\min\{\nu_{Q_j(F)}(z_0),M\}-\sum_{j=1}^q\omega_j\nu_{Q_j(F)}(z_0)\right)\\
&\ge\frac{\epsilon}{q(\sigma_M+\epsilon\tau_M)}\sum_{j=1}^q\min\{\nu_{Q_j(F)}(z_0),M\}\ge 0.
\end{align*}
  This implies that $d\tau^2$ is continuous pseudo-metric on $\Delta(R)$. 

We now prove that $d\tau^2$ has strictly negative curvature on $\Delta$. Let $\Omega$ be the Fubini-Study form of $\P^n(\C)$ and denote by $\Omega_f$ the pullback of $\Omega$ by the curve $f$. By  the inequality (\ref{new4}) and the definition of the function $h$, we see that
\begin{align*}
dd^c&\log\dfrac{|F_{\mathcal V,M}|}{\prod_{j=1}^q|Q_j(F)|^{\omega_j}}+\sum_{j=1}^q(\omega_j+\frac{\epsilon}{q})dd^c\log h_{j}\\
&\ge -\frac{1}{2}\sum_{j=1}^q\omega_j[\min\{M,\nu_{Q_j(f)}\}]+\frac{1}{2}\sum_{j=1}^q(\omega_j+\frac{\epsilon}{q})\left([\min\{M,\nu_{Q_j(f)}\}]+[v_j]\right)\ge 0.
\end{align*}
Then by Theorems \ref{thm3.3} and \ref{thm3.4} and the definition of $\eta$, we have
\begin{align}\label{new3}
\begin{split}
dd^c\log\eta&\ge\dfrac{\gamma-\epsilon(\sigma_{M+1}+\sum_{j=1}^q\frac{\eta_j}{q})}{\sigma_M+\epsilon\tau_M}d\Omega_f+\dfrac{\epsilon}{2(\sigma_M+\epsilon\tau_M)}dd^c\log\left(|F_{\mathcal V,0}|\cdots|F_{\mathcal V,M}|\right)\\
& +\dfrac{1}{2(\sigma_M+\epsilon\tau_M)}dd^c\log\dfrac{\prod_{p=0}^{M-1}|F_{\mathcal V,p}|^{2\epsilon}}{\prod_{p=0}^{M-1}\log^{4\omega_j}(\delta/\varphi_{\mathcal V,p}(Q_j))}\\
&\ge\dfrac{\epsilon\tau_M}{2\sigma_M(\sigma_M+\epsilon\tau_M)}\left(\dfrac{|F_{\mathcal V,0}|^2\cdots |F_{\mathcal V,M}|^2}{|F_{\mathcal V,0}|^{2\sigma_{M+1}}}\right)^{\frac{1}{\tau_M}}dd^c|z|^2\\
&+C_0\left (\dfrac{|F_{\mathcal V,0}|^{2\left(\tilde\omega (q-2N+k-1)-M+k\right)}|F_{\mathcal V,M}|^2}{\prod_{j=1}^q(|Q_j(F)|^2\prod_{p=0}^{M-1}\log^2(\delta/\varphi_{\mathcal V,p}(Q_j)))^{\omega_j}}\right)^{\frac{1}{\sigma_M}}dd^c|z|^2\\
&\ge C_1\left (\dfrac{|F_{\mathcal V,0}|^{\tilde\omega (q-2N+k-1)-M+k-\epsilon\sigma_{M+1}}|F_{\mathcal V,M}|\prod_{p=0}^M|F_{\mathcal V,p}|^\epsilon}{\prod_{j=1}^q(|Q_j(F)|\prod_{p=0}^{M-1}\log(\delta/\varphi_{\mathcal V,p}(Q_j)))^{\omega_j}}\right)^{\frac{2}{\sigma_M+\epsilon\tau_M}}dd^c|z|^2
\end{split}
\end{align}
for some positive constant $C_0,C_1$, where the last inequality comes from H\"{o}lder's inequality.
On the other hand, we have $|h_j|\le \|F\|^{d\eta_j}$,
\begin{align*}
|F_{\mathcal V,0}|^{\tilde\omega (q-2N+k-1)-M+k-\epsilon\sigma_{M+1}}&=|F_{\mathcal V,0}|^{\gamma-\epsilon\sigma_{M+1}+\sum_{j=1}^q\omega_j\eta_j}\\
&\ge |F_{\mathcal V,0}|^{\gamma-\epsilon(\sigma_{M+1}+\frac{\eta_j}{q})}h_1^{\omega_1+\frac{\epsilon}{q}}\cdots h_q^{\omega_q+\frac{\epsilon}{q}}.
\end{align*}
This implies that $\Delta \log\eta^2\ge C_2\eta^2$ for some positive constant $C_2$. Therefore, $d\tau^2$ has strictly negative curvature.
\end{proof}

Applying Lemma \ref{lem3.7}, we will prove the following main lemma of this paper.
\begin{lemma}[Main Lemma]\label{ML}
Let $V,\mathcal V, d, M,\{Q_i\}_{i=1}^q$ and $\{\omega_i\}_{i=1}^q$ be as in Theorem \ref{thm3.2}. Let $f:\Delta (R)\rightarrow V\subset\P^n(\C)$ be a holomorphic map with a reduced representation $F=(f_0,\ldots,f_n)$, which is nondegenerate over $I_d(V)$. Assume that there are positive real numbers $\eta_j$, divisor $\nu_j$, real-valued bounded functions $k_j\ (1\le j\le q)$ and a constant $c$ as in Lemma \ref{lem3.7}. Then for every $\epsilon >0$ satisfying 
$$\gamma=\sum_{j=1}^q\omega_j (1-\eta_j)-M-1>\epsilon(\sigma_{M+1}+\sum_{j=1}^q\frac{\eta_j}{q}),$$ 
there exists a positive constant $C$, depending only on $\mathcal V, Q_j,\eta_j,\omega_j\ (1\le j\le q)$, such that
\begin{align*}\eta&:=\dfrac{|F_{\mathcal V,0}|^{\gamma-\epsilon(\sigma_{M+1}+\sum_{j=1}^q\frac{\eta_j}{q})}\prod_{j=1}^qh_{j}^{\omega_j+\frac{\epsilon}{q}}|F_{\mathcal V,M}|^{1+\epsilon}\prod_{p=0}^{M-1}|F_{\mathcal V,p}(Q_j)|^{\epsilon/q}}{\prod_{j=1}^q|Q_j(F)|^{\omega_j}}\\
&\le C\left(\dfrac{2R}{R^2-|z|^2}\right)^{\sigma_M+\epsilon\tau_M},
\end{align*}
where $h_j=|k_j|\cdot\|F\|^{d\eta_j}\ (1\le j\le q)$.
\end{lemma}
\begin{proof}
As in the proof of Lemma \ref{lem3.7}, we have
$$dd^c\log\eta\le C_2\eta^2dd^c|z|^2.$$
According to Lemma \ref{lem3.6}, this implies that
$$ \eta\le C_3\dfrac{2R}{R^2-|z|^2},$$
for some positive constant. Then we have
$$ \left( \dfrac{|F_{\mathcal V,0}|^{\gamma-\epsilon(\sigma_{M+1}+\sum_{j=1}^q\frac{\eta_j}{q})}\prod_{j=1}^qh_{j}^{\omega_j+\frac{\epsilon}{q}}|F_{\mathcal V,M}|\prod_{p=0}^{M}|F_{\mathcal V,p}|^{\epsilon}}{\prod_{j=1}^q(|Q_j(F)|\prod_{p=0}^{M-1}\log (\delta/\varphi_{\mathcal V,p}(Q_j)))^{\omega_j}}\right )^{\frac{1}{\sigma_M+\epsilon\tau_M}}\le C_3\dfrac{2R}{R^2-|z|^2}.$$
It follows that
$$ \left( \dfrac{|F_{\mathcal V,0}|^{\gamma-\epsilon(\sigma_{M+1}+\sum_{j=1}^q\frac{\eta_j}{q})}\prod_{j=1}^qh_{j}^{\omega_j+\frac{\epsilon}{q}}|F_{\mathcal V,M}|^{1+\epsilon}\prod_{p=0}^{M-1}|F_{\mathcal V,p}(Q_j)|^{\epsilon/q}}{\prod_{j=1}^q|Q_j(F)|^{\omega_i}\prod_{j=1}^q\prod_{p=0}^{M-1}(\varphi_{\mathcal V,p}(Q_j))^{\epsilon/2q}(\log (\delta/\varphi_{\mathcal V,p}(Q_j)))^{\omega_j}}\right )^{\frac{1}{\sigma_M+\epsilon\tau_M}}\le C_3\dfrac{2R}{R^2-|z|^2}.$$
Note that the function $x^{\epsilon/q}\log^{\omega}\left (\dfrac{\delta}{x^2}\right)\ (\omega>0,0<x\le 1)$ is bounded. Then we have
$$ \left( \dfrac{|F_{\mathcal V,0}|^{\gamma-\epsilon(\sigma_{M+1}+\sum_{j=1}^q\frac{\eta_j}{q})}\prod_{j=1}^qh_{j}^{\omega_j+\frac{\epsilon}{q}}|F_{\mathcal V,M}|^{1+\epsilon}\prod_{p=0}^{M-1}|F_{\mathcal V,p}(Q_j)|^{\epsilon/q}}{\prod_{j=1}^q|Q_j(F)|^{\omega_i}}\right )^{\frac{1}{\sigma_M+\epsilon\tau_M}}\le C_4\dfrac{2R}{R^2-|z|^2},$$
for a positive constant $C_4$. The lemma is proved.
\end{proof}

\section{Modified defect relation for Gauss map with hypersurfaces}

In this section, we will prove the main theorems of the paper. Firstly, we need some following preparation.

\begin{lemma}[{cf. \cite[Lemma 1.6.7]{Fu93}}]\label{lem4.1}
Let $d\sigma^2$ be a conformal flat metric on an open Riemann surface $S$. Then for every point $p\in S$, there is a holomorphic and locally biholomorphic map $\Phi$ of a disk $\Delta (R_0):=\{w:|w| <R_0\}\ (0 <R_0\le\infty)$ onto an open neighborhood of $p$ with $\Phi(0)=p$ such that $\Phi$ is a local isometry, namely the pull-back $\Phi^*(d\sigma^2)$ is equal to the standard (flat) metric on $\Delta(R_0)$, and for some point $a_0$ with $|a_0|=1$, the curve $\Phi (\overline{0,R_0a_0})$ is divergent in $S$ (i.e. for any compact set $K\subset S$, there exists an $s_0<R_0$ such that $\Phi (\overline{0,s_0a_0})$ does not intersect $K$).
\end{lemma}

Let $x=(x_0,\ldots,x_{n}): S\rightarrow\R^{n+1}$ be the immersion of the minimal surface $S$ into $\R^{n+1}$. Let $(x,y)$ be a local isothermal coordinate of $S$. Then $z=x+iy$ is a local holomorphic coordinate of $S$. The generalized Gauss map $g$ of $x$ is defined (locally) by
$$ g:S\rightarrow\P^{n}(\C), g:=\left(\dfrac{\partial x_0}{\partial z}:\cdots:\dfrac{\partial x_{n}}{\partial z}\right).$$
Then, we denote by $G_z$ the local reduced representation of $g$ defined by
$$G_z=\left(\dfrac{\partial x_0}{\partial z},\ldots,\dfrac{\partial x_{n}}{\partial z}\right).$$
Then, with other local holomorphic coordinate $\xi$, we have 
$$G_\xi=G_z\cdot\left (\dfrac{dz}{d\xi}\right)$$
(this yields that the map $g$ is global well-defined). Also, the metric $ds^2$ on $S$ induced by the canonical metric on $\R^{n+1}$ satisfies
$$ ds^2=2\|G_z\|^2|dz|^2.$$
We note that, $g$ is holomorphic since $S$ is minimal.

Let $V$ be a $k$-dimension projective subvariety of $\P^n(\C)$ and $d$ is positive integer. We fix a $\C$-ordered basis $\mathcal V=([v_0],\ldots,[v_M])$ of $I_d(V)$ as in the Section 3. Assume that the image of $g$ is included in $V$ and $g$ is nondegenerate over $I_d(V)$.

Let $Q_1,\ldots,Q_q$ be $q$ hypersurfaces of the same degree $d$ of $\P^n(\C)$ in $N$-subgeneral position with respect to $V$ and $V\not\subset Q_i\ (1\le i\le q)$. Suppose that
$$ [Q_j]=\sum_{i=0}^Ma_{ji}[v_i],$$
where $\sum_{i=0}^M|a_{ji}|^2=1$.

Since $g$ is nondegenerate over $I_d(V)$, the contact function satisfying
$$ ((G_z)_{\mathcal V,p})_z(Q_j)\not\equiv 0, \forall 1\le j\le q,0\le p\le M,$$
here the first subscript $z$ mean $G_z$ is the local reduced representation of $g$ on a local holomorphic chat $(U,z)$, the second subscript $z$ means the derivatives are taken w.r.t. $z$.
Then, for each $j,p\ (1\le j\le q,0\le p\le M)$, we may choose $i_1,\ldots,i_p$ with $0\le i_1<\cdots<i_p\le M$ such that
$$ (\psi(G_z)_{jp})_z=\sum_{l\ne i_1,\ldots,i_p}a_{jl}W_z(v_l(G_z),v_{i_1}(G_z),\ldots,v_{i_p}(G_z))\not\equiv 0.$$
We also note that $(\psi(G_z)_{j0})_z=(G_z)_{\mathcal V,0}(Q_j)=Q_j(G_z)$ and $(\psi(G_z)_{jM})_z=((G_z)_{\mathcal V,M})_z$.

\begin{proof}[Proof of Theorem \ref{1.1}]
By replacing $Q_i$ with $Q_i^{d/\deg Q_i}\ (1\le i\le q)$ if necessary, we may assume that all $Q_i\ (1\le i\le q)$ are of the same degree $d$. 

From the assumption of the theorem, we have
$$\sum_{j=1}^qD^{M}_{g}(Q_j)>\dfrac{(2N-k+1)(M+1)}{k+1}+\dfrac{(2N-k+1)M(M+1)}{2d(k+1)}.$$
Then, there exist a compact subset $K\subset S$, numbers $\eta_j>0\ (1\le j\le q)$ such that
$$ \sum_{j=1}^q(1-\eta_j)> \dfrac{(2N-k+1)(M+1)(M+2d)}{2d(k+1)},$$
divisor $\nu_j\ (1\le j\le q)$ and bounded continuous functions $k_j$ (we may assume $0\le k_j<1$) with mild singularities on $S\setminus K$ such that $\nu_j\ge c'$ on $\supp\nu_j$ for a positive constant $c'$ and
$$ [\min\{\nu_{Q_j(f)},M\}]+[v_j]=d\eta_j\Omega_f +dd^c[\log k_j^2]$$
in the sense of currents. We set $h_{zj}=k_j\|G_z\|^{d\eta_j}$ for each local reduced representation $G_z$ of $g$. Then we have
$$ \nu_{h_{zj}}\ge c\text{ on }\supp\nu_{h_{zj}}, \text{ where }c=\min\{1,c'\}. $$

From Theorem \ref{lem2.3}, we have
$$ (q-2N+k-1)\tilde\omega=\sum_{j=1}^q\omega_j-k-1;\ \tilde\omega\ge\omega_j>0 \text{ and }\tilde\omega\ge\dfrac{k+1}{2N-k+1}.$$
Therefore,
\begin{align}\label{new2}
\begin{split}
\sum_{j=1}^q\omega_j(1-\eta_j)-M-1&\ge\tilde\omega(q-2N+k-1-\sum_{j=1}^q\eta_j)-M+k\\ 
&\ge\dfrac{k+1}{2N-k+1}\left (\sum_{j=1}^q(1-\eta_j)-2N+k-1\right)-M+k\\
&= \dfrac{k+1}{2N-k+1}\left (\sum_{j=1}^q(1-\eta_j)-\dfrac{(2N-k+1)(M+1)}{k+1}\right)\\
&>\dfrac{k+1}{2N-k+1}\cdot\dfrac{(2N-k+1)M(M+1)}{2d(k+1)}=\dfrac{\sigma_M}{d}.
\end{split}
\end{align}
Then, we can choose a rational number $\epsilon\ (>0)$ such that
$$ \dfrac{d(\sum_{j=1}^q\omega_j(1-\eta_j)-M-1)-\sigma_M}{d\sigma_{M+1}+\tau_M+d\sum_{j=1}^q\frac{\eta_j}{q}}>\epsilon> \dfrac{d(\sum_{j=1}^q\omega_j(1-\eta_j)-M-1)-\sigma_M}{\frac{c}{q}+d\sigma_{M+1}+\tau_M+d\sum_{j=1}^q\frac{\eta_j}{q}}.$$
We define the following numbers
\begin{align*}
h&:=d(\sum_{j=1}^q\omega_j(1-\eta_j)-M-1)-\epsilon\left (d\sigma_{M+1}+d\sum_{j=1}^q\frac{\eta_j}{q}\right)>\sigma_M+\epsilon\tau_{M},\\ 
\rho&:=\dfrac{1}{h}(\sigma_M+\epsilon\tau_M),\\
\rho^*&:=\dfrac{1}{(1-\rho)h}=\dfrac{1}{d(\sum_{j=1}^q\omega_j(1-\eta_j)-M-1)-\sigma_M-\epsilon(d\sigma_{M+1}+\tau_M+d\sum_{j=1}^q\frac{\eta_j}{q})}. 
\end{align*}
It is clear that $0<\rho<1$ and $\frac{c\epsilon\rho^*}{q}>1.$
We consider a set
$$ A_1=\{a\in A; (\psi (G_z)_{jp})_z(a)\ne 0,h_{zj}(a)\ne 0\ \forall j=1,\ldots,q; p=0,\ldots,M\},$$
where $A=S\setminus K$, $z$ is a local holomorphic coordinate around $a$, and define a new pseudo-metric on $A_1$ as follows
$$ d\tau^2=\left (\dfrac{\prod_{j=1}^q\|G_z(Q_j)\|^{\omega_j}}{|((G_z)_{\mathcal V,M})_z)|^{1+\epsilon}\prod_{j=1}^qh_{zj}^{\omega_j+\frac{\epsilon}{q}}\prod_{j,p}|(\psi(G_z)_{jp})_z|^{\frac{\epsilon}{q}}}\right)^{2\rho^*}|dz|^2,$$
where $h_{zj}=k_j\|F_z\|^{d\eta_j}$. Here and throughout this section, for simplicity we will write $\prod_{j,p}$ for $\prod_{j=1}^q\prod_{p=1}^{M-1}$.
We note that the definition of $A_1$ and $d\tau^2$ do not depend on the choice of the local holomoprhic coordinate.
 Indeed, we have
\begin{align*}
&G_z=(\xi'_z)\cdot G_\xi,\ Q_j(G_z)=(\xi'_z)^d\cdot Q_j(G_\xi),\ v_j(G_z)=(\xi'_z)^d\cdot v_j(G_\xi),\ h_{zj}&=|\xi'_z|^{d\eta_j}\cdot h_{\xi j}\\
&((G_z)_{\mathcal V,p})_z=(\xi'_z)^{d(p+1)}\cdot((G_\xi)_{\mathcal V,p})_z=(\xi'_z)^{d(p+1)+\sigma_p}\cdot((G_\xi)_{\mathcal V,p})_\xi\\
&(\psi (G_z)_{ip})_z=(\xi'_z)^{d(p+1)+\sigma_p}\cdot(\psi(G_z)_{ip})_\xi\ (0\le p\le M).\\
\end{align*}
and hence $(\psi (G_z)_{ip})_z(a)=0$ if and only if $(\psi (G_\xi)_{ip})_\xi(a)=0$ for two local holomorphic coordinates $z$ and $\xi$ around $a$. On the other hand, we have

\begin{align*}
&\left (\dfrac{\prod_{j=1}^q|Q_j(G_z)|^{\omega_j}}{|((G_z)_{\mathcal V,M})_z)|^{1+\epsilon}\prod_{j=1}^qh_{zj}^{\omega_j+\frac{\epsilon}{q}}\prod_{j,p}|(\psi(G_z)_{jp})_z|^{\frac{\epsilon}{q}}}\right)^{2\rho^*}|dz|^2\\ 
=& \left (\dfrac{\left(\prod_{j=1}^q|Q_j(G_\xi)|^{\omega_j}\right)|\xi'_z|^{d\sum_{j=1}^q\omega_j(1-\eta_j)-d\epsilon\sum_{j=1}^q\frac{\eta_j}{q}}}{|((G_\xi)_{\mathcal V,M})_\xi)|^{1+\epsilon}\prod_{j=1}^qh_{\xi j}^{\omega_j+\frac{\epsilon}{q}}\prod_{j,p}|(\psi(G_\xi)_{jp})_\xi|^{\frac{\epsilon}{q}}|\xi'_z|^{d(M+1)+\sigma_M+\epsilon\sum_{p=0}^{M}(d(p+1)+\sigma_p)}}\right)^{2\rho^*}|dz|^2\\
=& \left (\dfrac{\left(\prod_{j=1}^q|Q_j(G_\xi)|^{\omega_j}\right)|\xi'_z|^{d\sum_{j=1}^q\omega_j(1-\eta_j)-d\epsilon\sum_{j=1}^q\frac{\eta_j}{q}}}{\left(|((G_\xi)_{\mathcal V,M})_\xi)|^{1+\epsilon}\prod_{j=1}^qh_{\xi j}^{\omega_j+\frac{\epsilon}{q}}\prod_{j,p}|(\psi(G_\xi)_{jp})_\xi|^{\frac{\epsilon}{q}}\right)|\xi'_z|^{d(M+1)+\sigma_M+\epsilon(d\sigma_{M+1}+\tau_{M})}}\right)^{2\rho^*}|dz|^2\\
=& \left (\dfrac{\prod_{j=1}^q|Q_j(G_\xi)|^{\omega_j}}{|((G_\xi)_{\mathcal V,M})_\xi)|^{1+\epsilon}\prod_{j=1}^qh_{\xi j}^{\omega_j+\frac{\epsilon}{q}}\prod_{j,p}|(\psi(G_\xi)_{jp})_\xi|^{\frac{\epsilon}{q}}}\right)^{2\rho^*}|d\xi|^2
\end{align*}
(since $|\xi'_z|^{2\rho^*(d(\sum_{j=1}^q\omega_j(1-\eta_j)-M-1)-\sigma_M-\epsilon(d\sigma_{M+1}+\tau_M+d\sum_{j=1}^q\frac{\eta_j}{q}))}|dz^2|=|\xi'_z|^2|dz^2|=|d\xi^2|$).
Then the pseudo-metric $d\tau^2$ is global well-defined.


Since $Q_j(G_z),((G_z)_{\mathcal V,M})_z,\psi(G_z)_{jp})_z$ are all holomorphic functions and $u_j\ (1\le j\le q)$ are harmonic functions on $A_1$, $d\tau$ is flat on $A_1$. So, $d\tau$ can be smoothly extended over $K$. Thus, we have a metric, still call it $d\tau$, on
$$A_1' =A_1\cup K$$
that is flat outside the compact set $K$.

\begin{claim}\label{cl4.2}
$d\tau$ is complete on $A_1'$
\end{claim}
Indeed, suppose contrarily that $A_1'$ is not complete with $d\tau$, there is a divergent curve $\gamma: [0,1)\rightarrow A_1'$ with finite length. By using only the last segment of $\gamma$, we may assume that the distance $d(\gamma,K)$ between $\gamma$ and $K$ exceeds a positive number $d$. Then, as $t\rightarrow 1$ there are only two cases: either $\gamma(t)$ tends to a point $a$ with
$$ \prod_{j=1}^q(h_{zj}\prod_{p=0}^M|\psi(G_z)_{ip}|)(a)=0$$
for a local holomorphic coordinate $z$ around $a$, or else $\gamma (t)$ tends to the boundary of $S$.

For the first case, by Theorem \ref{thm3.5}, we have
\begin{align*}
\nu_{d\tau}(a)&\le -\biggl (\nu_{((G_z)_{\mathcal V,M})_z}(a)-\sum_{j=1}^q\omega_j\nu_{G_z(Q_j)}(a)+\sum_{j=1}^q\omega_j\min\{\nu_{G_z(Q_j)}(a),M\}\\ 
&\ \ +\bigl(\epsilon\nu_{((G_z)_{\mathcal V,M})_z(a)}+\dfrac{\epsilon}{q}\sum_{j=1}^q\big(\nu_{h_{zj}}+\sum_{p=0}^{M-1}\nu_{\psi(G_z)_{jp}}(a)\big)\bigl)\biggl)\rho^*\\
&\le -\epsilon\rho^*\nu_{((G_z)_{\mathcal V,M})_z}(a)-\dfrac{\epsilon\rho^*}{q}\sum_{j=1}^q\big(\nu_{h_{zj}}+\sum_{p=0}^{M-1}\nu_{\psi(G_z)_{jp}}(a)\big)\le -\dfrac{c\epsilon\rho^*}{q}
\end{align*}
(since $\nu_{h_{zj}}>c$ on $\supp\nu_j$). Then, there is a positive constant $C$ such that
$$ |d\tau|\ge\dfrac{C}{|z-z(a)|^{\frac{\epsilon\rho^*}{q}}}|dz|$$
in a neighborhood of $a$. Then we get
$$ L_{d\tau}(\gamma)=\int_0^1\|\gamma'(t)\|_{\tau}dt=\int_{\gamma}d\tau \ge\int_\gamma \dfrac{C}{|z-z(a)|^{\frac{c\epsilon\rho^*}{q}}}|dz|=+\infty$$
($\gamma (t)$ tends to $a$ as $t\rightarrow 1$). This is a contradiction. Then, the second case must occur, that is $\gamma (t)$ tends to the boundary of $S$ as $t\rightarrow 1$.  

Choose $t_0$ close $1$ enough such that $L_{d\tau}(\rho|_{(t_0,1)})<d/3$. Take a disk $\Delta$ (in the metric induced by $d\tau$) around $\gamma(t_0)$. Since $d\tau$ is flat, by Lemma \ref{lem4.1}, $\Delta$ is isometric to an ordinary disk in the plane. Let $\Phi:\Delta(r)=\{|\omega|<r\}\rightarrow\Delta$ be this isometric with $\Phi(0)=\gamma(t_0)$. Extend $\Phi$ as a local isometric into $A_1$ to a the largest disk possible $\Delta(R)=\{|\omega|<R\}$, and denoted again by $\Phi$ this extension (for simplicity, we may consider $\Phi$ as the exponential map).  Since $\int_{\gamma|_{[t_0,1)}}d\tau<d/3$, we have $R\le d/3$. Hence, the image of $\Phi$ is bounded away from $K$ by distance at least $2d/3$. Since $\Phi$ cannot be extended to a larger disk, it must be hold that the image of $\Phi$ goes to the boundary $A_1$. But, this image cannot go to points $z_0$ of $A'_1$ with $\prod_{j=1}^qh_{zj}(z_0)\prod_{p=0}^M|\psi(G_z)_{ip}|(z_0)=0$, since we have already shown that $\gamma(t_0)$ is infinitely far away in the metric with respect to these points. Then the image of $\Phi$ must go to the boundary $S$. Hence, by again Lemma \ref{lem4.1}, there exists a point $w_0$ with $|w_0|= R$ so that $\Gamma=\Phi(\overline{0,w_0})$ is a divergent curve on $S$.

Our goal now is to show that $\Gamma$ has finite length in the original metric $ds^2$ on $S$, contradicting the completeness of $S$. Let $f:=g\circ\Phi :\Delta(R)\rightarrow V\subset\P^n(\C)$ be a holomorphic curve which is nondegenerate over $I_d(V)$. Let $z$ be a local holomorphic coordinate on the image of $\Phi$, then $f$ have a local reduced representation
$$ F=(f_0,\ldots,f_n),$$
where $f_i=G_z\circ\Phi\ (0\le i\le n).$
Hence, we have locally:
\begin{align*}
\Phi^*ds^2&=2\|G_z\circ\Phi\|^2|\Phi^*dz|^2=2\|F\|^2\left|\dfrac{d(z\circ\Phi)}{dw}\right|^2|dw|^2,\\
(F_{\mathcal V,M})_w&=((G_z\circ\Phi)_{\mathcal V,M})_w=((G_z)_{\mathcal V,M})_z\circ\Phi\cdot\left(\dfrac{d(z\circ\Phi)}{dw}\right)^{\sigma_M},\\
(\psi(F)_{jp})_w&=(\psi(G_z\circ\Phi)_{jp})_w=(\psi(G_z)_{jp})_z\cdot\left(\dfrac{d(z\circ\Phi)}{dw}\right)^{\sigma_p}, (0\le p\le M).
\end{align*}
On the other hand, since $\Phi$ is locally isometric,
\begin{align*}
&|dw|=|\Phi^*d\tau|\\
&=\left (\dfrac{\prod_{j=1}^q|Q_j(G_z)\circ\Phi|^{\omega_j}}{|((G_z)_{\mathcal V,M})_z\circ\Phi)|^{1+\epsilon}\prod_{j=1}^q(h_{zj}\circ\Phi)^{\omega_j+\frac{\epsilon}{q}}\prod_{j,p}|(\psi(G_z)_{jp})_z\circ\Phi|^{\epsilon/q}}\right)^{\rho^*}\left|\dfrac{d(z\circ\Phi)}{dw}\right|\cdot|dw|\\
&=\left (\dfrac{\prod_{j=1}^q|Q_j(F)|^{\omega_j}}{|(F_{\mathcal V,M})_w|^{1+\epsilon}\prod_{j=1}^q(h_{zj}\circ\Phi)^{\omega_j+\frac{\epsilon}{q}}\prod_{j,p}|(\psi(F)_{jp})_w|^{\epsilon/q}}\right)^{\rho^*}\left|\dfrac{d(z\circ\Phi)}{dw}\right|^{1+h\rho\rho^*}\cdot|dw|
\end{align*}
(because $1+\rho^*(\sigma_M+\epsilon\tau_M)=1+h\rho\rho^*$).
This implies that
\begin{align*}
\left|\dfrac{d(z\circ\Phi)}{dw}\right|&=\left (\dfrac{|(F_{\mathcal V,M})_w|^{1+\epsilon}\prod_{j=1}^q(h_{zj}\circ\Phi)^{\omega_j+\frac{\epsilon}{q}}\prod_{j,p}|(\psi(F)_{jp})_w|^{\epsilon/q}}{\prod_{j=1}^q|Q_j(F)|^{\omega_j}}\right)^{\frac{\rho^*}{1+h\rho\rho^*}}\\
&\le \left (\dfrac{|(F_{\mathcal V,M})_w|^{1+\epsilon}\prod_{j=1}^q(h_{zj}\circ\Phi)^{\omega_j+\frac{\epsilon}{q}}\prod_{j,p}|(F_{\mathcal V,p})_w(Q_j)|^{\epsilon/q}}{\prod_{j=1}^q|Q_j(F)|^{\omega_j}}\right)^{\frac{\rho^*}{1+h\rho\rho^*}}\\
&=\left (\dfrac{|(F_{\mathcal V,M})_w|^{1+\epsilon}\prod_{j=1}^q(h_{zj}\circ\Phi)^{\omega_j+\frac{\epsilon}{q}}\prod_{j,p}|(F_{\mathcal V,p})_w(Q_j)|^{\epsilon/q}}{\prod_{j=1}^q|Q_j(F)|^{\omega_j}}\right)^{1/h}.
\end{align*}
Hence, we have
\begin{align*}
\Phi^*ds&\le\sqrt{2}\|F\|\left (\dfrac{|(F_{\mathcal V,M})_w|^{1+\epsilon}\prod_{j=1}^q(h_{zj}\circ\Phi)^{\omega_j+\frac{\epsilon}{q}}\prod_{j,p}|(F_{\mathcal V,p})_w(Q_j)|^{\epsilon/q}}{\prod_{j=1}^q|Q_j(F)|^{\omega_j}}\right)^{1/h}|dw|\\
&=\sqrt{2}\left (\dfrac{|F_{\mathcal V,0}|^{\frac{h}{d}}(F_{\mathcal V,M})_w|^{1+\epsilon}\prod_{j=1}^q(h_{zj}\circ\Phi)^{\omega_j+\frac{\epsilon}{q}}\prod_{j,p}|(F_{\mathcal V,p})_w(Q_j)|^{\epsilon/q}}{\prod_{j=1}^q|Q_j(F)|^{\omega_j}}\right)^{1/h}|dw|.
\end{align*}
We note that $\frac{h}{d}=\sum_{j=1}^q\omega_j(1-\eta_j)-M-1-\epsilon(\sigma_{M+1}+\sum_{j=1}^q\frac{\eta_j}{q})$. Then the inequality (\ref{new2}) yields that the conditions of Lemma \ref{ML} are satisfied. Then, by applying Lemma \ref{ML} we have
$$ \Phi^*ds\le C\left (\dfrac{2R}{R^2-|w|^2}\right)^\rho|dw|,$$
for some positive constant $C$. Also, we have that $0<\rho<1$. Then
$$L_{ds^2}(\Gamma)\le\int_{\Gamma}ds=\int_{\overline{0,w_0}}\Phi^*ds\le C\cdot\int_{0}^R\left(\dfrac{2R}{R^2-|w|^2}\right)^{\rho}|dw|<+\infty. $$
This contradicts the assumption of completeness of $S$ with respect to $ds^2$. Thus, Claim \ref{cl4.2} is proved.

Then, we note that the metric $d\tau^2$ on $A_1'$ is flat outside $K$, and $A_1'$ is complete with this metric by Claim \ref{cl4.2}. Also, as we known, a theorem of Huber (cf. \cite[Theorem 13, p.61]{Hu61}) yields that if $A_1'$ has finite total curvature then $A_1'$ is finitely connected. This means that there is a compact subregion of $A_1'$ whose complement is the union of a finite number of doubly-connected regions. Therefore,  the functions $\prod_{p=0}^M\prod_{j=1}^q|\psi(G_z)_{jp}|$ must have only a finite number of zeros, and the original surface $S$ is finitely connected. Due to Osserman (cf. \cite[Theorem 2.1]{O63}), each annular ends of $A_1'$, and hence of $S$, is conformally equivalent to a punctured disk. Thus, the Riemann surface $S$ must be conformally equivalent to a compact Riemann surface $\overline{S}$ with a finite number of points removed. By the assumption, in a neighborhood of each of those points the Gauss map $g$ satisfies
$$ \sum_{j=1}^q\delta_{g}^{H}(Q_j)> \dfrac{(2N-k+1)(M+1)}{k+1}.$$
Now by Lemma \ref{cor2.9}, the Gauss map $g$ is not essential at those points. Therefore $g$ can be extended to a holomorphic map from $\overline{S}$ to $V\subset\P^n(\C)$. If the homology class represented by the image of $g:S\rightarrow V\subset\P^n(\C)$ is $l$ times the fundamental homology class of $\P^n(\C)$, then we have
$$\iint K(s)dS=-2\pi l $$
as the total curvature of $S$, where $K(s)$ is the Gaussian curvature of $S$. Hence, Theorem \ref{1.1} is proved.
\end{proof}

\begin{proof}[Proof of Theorem \ref{1.3}]
By Replacing $Q_i$ by $Q_i^{d/\deg Q_j}$ if necessary, we may assume that $\deg Q_j=d\ (1\le j\le q)$. We fix an ordered basis $\mathcal V$ for $I_d(V)$. We take a global reduced representation $G=(g_0,\ldots,g_n)$ of $g$. Suppose contrarily that
$$\sum_{j=1}^q\delta^{S,M}_{g,S}(Q_j)> \dfrac{(2N-k+1)(M+1)}{k+1}+\dfrac{(2N-k+1)M(M+1)}{2d(k+1)}.$$
Then for each $j\ (1\le j\le q)$, there is a constant $\eta_j$ and a continuous subharmonic functions $u_j$ on $S$ satisfying conditions $(D_1),(D_2)$ with respect to the number $\eta_j$, the reduced representation $G$ and the hypersurface $Q_j$ such that
$$\sum_{j=1}^q(1-\eta_j)> \dfrac{(2N-k+1)(M+1)}{k+1}+\dfrac{(2N-k+1)M(M+1)}{2d(k+1)}.$$

Now, for a (local) holomorphic reduced representation $G_z$ of $g$ on a holomorphic chat $(U,z)$, there is a non-vanishing holomorphic function $\theta_z$ such that $G_z=\theta_z\cdot G$. We define 
$$ h_{zj}=|\theta_z|^{d\eta_j}e^{u_j}\le |\theta_z|^{d\eta_j}\cdot\|G\|^{d\eta_j}=\|G_z\|^{d\eta_j}.$$
It easily see that if $\xi$ is another local holomorphic coordinate then $\theta_z=\theta_\xi\cdot\xi'_z$, and hence $h_{zj}=|\xi'_z|^{d\eta_j}h_{\xi j}$. 

From (\ref{new2}), we have
$$\gamma:= d(\sum_{j=1}^q\omega_j(1-\eta_j)-M-1)>\sigma_M.$$
Then, we may choose a rational number $\epsilon\ (>0)$ such that
$$\gamma-\sigma_M=\epsilon (\sigma_{M+1}+\tau_M).$$
We set
$$ \lambda_z=\left( \dfrac{|(G_z)_{\mathcal V,0}|^{\gamma-\epsilon\sigma_{M+1}}\prod_{j=1}^qh_{zj}^{\omega_j}|(G_z)_{\mathcal V,M}|\prod_{p=0}^{M}|(G_z)_{\mathcal V,p}|^{\epsilon}}{\prod_{j=1}^q(|Q_j(G_z)|\prod_{p=0}^{M-1}\log (\delta/\varphi_{z,p}(Q_j)))^{\omega_j}}\right )^{\frac{1}{\sigma_M+\epsilon\tau_M}},$$
where $\varphi_{z,p}(Q_j)=\dfrac{|((G_z)_{\mathcal V,p})_z(Q)|}{|((G_z)_{\mathcal V,p})_z|}$ and $\delta$ is the number satisfying the conclusion of Theorem \ref{thm3.3} with $G_z,\varphi_{z,p}(Q_j)$ in place of $F,\varphi_{\mathcal V,p}(Q_j)$.

We define a pseudo-metric $d\tau^2:=\lambda_z^2|dz|^2$ on a holomorphic chat $(U,z)$ of $S$. We will show that $d\tau^2$ is well-defined, i.e., not depend the choice of the local coordinate. Indeed, if we have another local coordinate $\xi$, we have
\begin{align*}
\lambda_{z}&=\left( \dfrac{|(G_z)_{\mathcal V,0}|^{\gamma-\epsilon\sigma_{M+1}}\prod_{j=1}^qh_{zj}^{\omega_j}|(G_z)_{\mathcal V,M}|\prod_{p=0}^{M}|(G_z)_{\mathcal V,p}|^{\epsilon}}{\prod_{j=1}^q(|Q_j(G_z)|\prod_{p=0}^{M-1}\log (\delta/\varphi_{z,p}(Q_j)))^{\omega_j}}\right )^{\frac{1}{\sigma_M+\epsilon\tau_M}}\\
&=\left( \dfrac{|(G_\xi)_{\mathcal V,0}|^{\gamma-\epsilon\sigma_{M+1}}\prod_{j=1}^qh_{\xi j}^{\omega_j}|(G_\xi)_{\mathcal V,M}| |\xi'_z|^{d\sigma_M}\prod_{p=0}^{M}|(G_\xi)_{\mathcal V,p}|^{\epsilon}|\xi'_z|^{d\epsilon\sum_{p=0}^M\sigma_p)}}{\prod_{j=1}^q(|Q_j(G_\xi)|\prod_{p=0}^{M-1}\log (\delta/\varphi_{\xi,p}(Q_j)))^{\omega_j}}\right )^{\frac{1}{\sigma_M+\epsilon\tau_M}}\\
&=\left( \dfrac{|(G_\xi)_{\mathcal V,0}|^{\gamma-\epsilon\sigma_{M+1}}\prod_{j=1}^qh_{\xi j}^{\omega_j}|(G_\xi)_{\mathcal V,M}|\prod_{p=0}^{M}|(G_\xi)_{\mathcal V,p}|^{\epsilon}}{\prod_{j=1}^q(|Q_j(G_\xi)|\prod_{p=0}^{M-1}\log (\delta/\varphi_{\xi,p}(Q_j)))^{\omega_j}}\right )^{\frac{1}{\sigma_M+\epsilon\tau_M}}|\xi'_z|.
\end{align*}
Hence $\lambda_z^2|dz^2|=\lambda_\xi^2|\xi'_z|^2|dz^2|=\lambda_\xi^2|d\xi^2|$. Then, $d\tau^2$ is well-defined.

It is obvious that $d\tau^2$ is continuous on $S\setminus\bigcup_{j=1}^q g^{-1}(Q_j)$. Take a point $a$ such that $\prod_{j=1}^qQ_j(G_z)=0$ (for a loal holomorphic coordinate $z$ around $a$). We have
$$ \lim_{t\rightarrow a}(h_{zj}(t)\cdot |z(t)-z(a)|^{-\min\{M,\nu_{Q_j(G_z)}(a)\}})<\infty.$$
Combining with Theorem \ref{thm3.5}, this implies that
\begin{align*}
d\tau_z(a)&\ge\dfrac{1}{\sigma_M+\epsilon\tau_M}\biggl (\nu_{(G_z)_{\mathcal V,M}}+\sum_{j=1}^q\omega_j(\min\{M,\nu_{Q_j(G_z)}(a)\}-\nu_{Q_j(G_z)}(a))\biggl)\\ 
& \ge 0.
\end{align*}
Therefore $d\tau$ is continuous at $a$. This yields that $d\tau$ is a continuous pseudo-metric on $S$.

We now prove that $d\tau^2$ has strictly negative curvature on $S$. By Theorem \ref{thm3.5}, we see that
\begin{align*}
dd^c\left[\log\dfrac{|(G_z)_{\mathcal V,M}|}{\prod_{j=1}^q|Q_j(G_z)|^{\omega_j}}\right]&+\sum_{j=1}^q\omega_jdd^c[\log h_{zj}]\\
&\ge \sum_{i=1}^q(-[\min\{M,\nu_{Q_j(f)}\}]+dd^c[\log h_{zj}])\ge 0.
\end{align*}
Then by Theorems \ref{thm3.3} and \ref{thm3.4}, similarly as (\ref{new3}) we have
\begin{align*}
dd^c\log\lambda_z&\ge\dfrac{\gamma-\epsilon\sigma_{M+1}}{\sigma_M+\epsilon\tau_M}d\Omega_g+\dfrac{\epsilon}{2(\sigma_M+\epsilon\tau_M)}dd^c\log\left(|(G_z)_{\mathcal V,0}|\cdots|(G_z)_{\mathcal V,M}|\right)\\
& +\dfrac{1}{2(\sigma_M+\epsilon\tau_M)}dd^c\log\dfrac{\prod_{p=0}^{M-1}|(G_z)_{\mathcal V,p}|^{2\epsilon}}{\prod_{p=0}^{M-1}\log^{4\omega_j}(\delta/\varphi_{\mathcal V,p}(Q_j))}\\
&\ge C\left (\dfrac{|(G_z)_{\mathcal V,0}|^{\tilde\omega (q-2N+k-1)-M+k-\epsilon\sigma_{M+1}}|(G_z)_{\mathcal V,M}|\prod_{p=0}^M|(G_z)_{\mathcal V,p}|^\epsilon}{\prod_{j=1}^q(|Q_j(F)|\prod_{p=0}^{M-1}\log(\delta/\varphi_{z,p}(Q_j)))^{\omega_j}}\right)^{\frac{2}{\sigma_M+\epsilon\tau_M}}dd^c|z|^2
\end{align*}
for some positive constant $C$. On the other hand, we have $|h_{zj}|\le\|G_z\|^{d\eta_j}$,
\begin{align*}
|(G_z)_{\mathcal V,0}|^{\tilde\omega (q-2N+k-1)-M+k-\epsilon\sigma_{M+1}}
&=|(G_z)_{\mathcal V,0}|^{\gamma-\epsilon\sigma_{M+1}+\sum_{j=1}^q\omega_j\eta_j}\\
&\ge |(G_z)_{\mathcal V,0}|^{\gamma-\epsilon\sigma_{M+1}}h_1^{\omega_1}\cdots h_q^{\omega_q}.
\end{align*}
This implies that $\Delta \log\eta^2\ge C\eta^2$. Therefore, $d\tau^2$ has strictly negative curvature.

As we known that the universal covering surface of $S$ is biholomorphic to $\C$ or a disk in $\C$. If the universal covering of $S$ is holomorphic to $\C$ then from \cite[Theorem 2]{QA} we have
$$\sum_{j=1}^q\delta^{S,M}_{g,S}\ge\sum_{j=1}^q\delta_g^M(Q_j)\le \dfrac{(2N-k+1)(M+1)}{k+1},$$
which contradicts the supposition. Now, we consider the case where the universal covering surface $\tilde S$ of $S$ is biholomorphic to the unit disc $\Delta$. Without loss of generality, we may assume that $\tilde S=\Delta$ and have the covering map $\Phi :\Delta\rightarrow S$. From Lemma \ref{lem3.6} we have
$$ \Phi^*d\tau^2\le d\sigma_{\Delta}^2,$$
where $d\sigma_{\Delta}=\dfrac{2}{1-|z|^2}|dz|^2$ with the complex coordinate $z$ on $\Delta$.

Since $S$ has total finite curvature, $S$ is is conformally equivalent to a compact surface $\overline S$ punctured at a finite number of points $P_1,\ldots,P_r$. Then,  there are disjoint neighborhoods $U_i$ of $P_i\ (1\le i\le r)$ in $\overline S$ and biholomorphic maps $\phi_i:U_i\rightarrow\Delta$ with $\phi_i(P_i)=0$. Note that, the Poincare-metric on $\Delta^*=\Delta\setminus\{0\}$ is given by $d\sigma_{\Delta^*}^2=\dfrac{4|dz|^2}{|z|^2\log^2|z|^2}$. Hence, by the decreasing distance property, we have
$$\Phi^*d\tau^2\le d\sigma_\Delta^2\le C\cdot(\Phi\circ\phi_i^{-1})^*d\sigma_{\Delta^*}^2\ (1\le i\le r)$$
for some positive constant $C$. This implies that
$$\int_{U_i}\Omega_{d\tau^2}\le \int_{\Phi^{-1}(U_i)}\Phi^*\Omega_{\sigma_\Delta^2}\le lC\int_{\Delta^*}\Omega_{d\sigma_{\Delta^*}^2}<\infty.$$
where $l$ is the number of the sheets of the covering $\Phi$. Then, it yields that
$$\int_S\Omega_{d\tau^2}\le \int_{S\setminus\bigcup_{i=1}^rU_i}\Omega_{d\tau^2}+lCr\int_{\Delta^*}\Omega_{d\sigma_{\Delta^*}^2}<\infty.$$

Now, denote by $ds^2$ the original metric on $S$. Similar as (\ref{new3}), we have
$$ dd^c\log\lambda_z \ge\dfrac{\gamma-\epsilon\sigma_{M+1}}{\sigma_M+\epsilon\tau_M}d\Omega_g.$$ 
Then there is a subharmonic function $v_z$ such that
\begin{align*}
\lambda_z^2|dz|^2&=e^{v_z}\|G_z\|^{2\frac{\gamma-\epsilon\sigma_{M+1}}{\sigma_M+\epsilon\tau_M}}|dz|^2\\
&=e^{v_z}\|G_z\|^{2\frac{\gamma-\sigma_M-\epsilon\tau_{M+1}}{\sigma_M+\epsilon\tau_M}}\|G_z\|^2 |dz|^2\\
&=e^w ds^2
\end{align*}
for a subharmonic function $w$ on $S$. Since $S$ is complete with respect to $ds^2$, applying a result of S. T. Yau \cite{Y76} and L. Karp \cite{K82} we have
$$ \int_S\Omega_{d\tau^2}=\int_Se^w\Omega_{ds^2}=+\infty.$$
This is a contradiction.

Then we must have
$$\sum_{j=1}^q\delta^{S,M}_{g,S}(Q_j)\le \dfrac{(2N-k+1)(M+1)}{k+1}+\dfrac{(2N-k+1)M(M+1)}{2d(k+1)}.$$
The theorem is proved.
\end{proof}

\noindent{\bf Acknowledgements.} This work was completed when the author was visiting at Institute of Mathematics, Vietnamese Academy of Science and Technology (VAST) with the support of The Targeted Grants to Institutes program of Simon Foundation, under grant number 558672. We would like to thank the Institute of Mathematics for the kind hospitality and thank Simon Foundation for the support. This research is funded by Vietnam National Foundation for Science and Technology Development (NAFOSTED) under grant number 101.02-2021.12.


\begin{thebibliography}{99}

\bibitem{A38} L.V. Ahlfors, \textit{An extension of Schwarz's lemma, Trans. Am. Math. Soc}. \textbf{43} (1938) 359--364.

\bibitem{Ha18a} P. H. Ha, \textit{Modified defect relations of the Gauss map and the total curvature of a complete minimal surface,} Topology Appl. \textbf{234} (2018), 178--197.

\bibitem{Ha18b} P. H. Ha, \textit{Non-integrated defect relations for the Gauss map of a complete minimal surface with finite total curvature in $\R^m$}, Bull. Math. Soc. Sci. Math. Roum., Nouv. Sér. \textbf{61}(109), no. 2 (2018), 187--199.

\bibitem{HTP} P.H. Ha, L.B. Phuong, P.D. Thoan, \textit{Ramification of the Gauss map and the total curvature of a complete minimal surface}, Topol. Appl. \textbf{199} (2016), 32--48.

\bibitem{Hu61} A. Huber, \textit{On subharmonic functions and differential geometry in large}, Comment. Math. Helv., \textbf{32} (1961), 13--72.

\bibitem{Fu83} H. Fujimoto, \textit{Value distribution of the Gauss maps of complete minimal surfaces in $\R^m$}, J. Math. Soc. Japan, \textbf{35}, no. 4 (1983), 663--681.

\bibitem{Fu83b} H. Fujimoto, \textit{On the Gauss map of a complete minimal surface in $\R^m$} , J. Math. Soc. Japan, \textbf{35} (1983), 279--288.

\bibitem{Fu85} H. Fujimoto, \textit{Non-integrated defect relation for meromorphic maps of complete K\"{a}hler manifolds into $\P^N(\C)\times\cdots\times\P^N(\C)$}, Jpn. J. Math. \textbf{11} (1985), 233--264.

\bibitem{Fu88} H. Fujimoto, \textit{On the number of exceptional values of the Gauss maps of minimal surfaces}, J. Math. Soc. Japan, \textbf{40} (1988), 235--247.

\bibitem{Fu89} H. Fujimoto, \textit{Modified defect relations for the Gauss map of minimal surfaces}, J. Differential Geom., \textbf{29} (1989), 245--262.

\bibitem{Fu90} H. Fujimoto, \textit{Modified defect relations for the Gauss map of minimal surfaces II}, J. Differential Geom., \textbf{31} (1990), 365--385.

\bibitem{Fu91} H. Fujimoto, \textit{Modified defect relations for the Gauss map of minimal surfaces III}, Nagoya Math. J., \textbf{124} (1991), 13--40.


\bibitem{Fu93} H. Fujimoto, \textit{Value Distribution Theory of the Gauss map of Minimal Surfaces in $\R^m$}, Aspect of Math., Vol. E21, Vieweg, Wiesbaden (1993).

\bibitem{K82} L. Karp, \textit{Subharmonic functions on real and complex manifolds}, Math. Z. \textbf{179} (1982) 535--554. 

\bibitem{RJ} L. Jin, M. Ru, \textit{Algebraic curves and the Gauss map of algebraic minimal surfaces}, Differ. Geom. Appl. \textbf{25} (2007) 701--712.

\bibitem{MO90} X. Mo and R. Osserman, \textit{On the Gauss map and total curvature of complete minimal surfaces and an extension of Fujimoto's theorem,} J. Differential Geom., \textbf{31} (1990), 343--355.

\bibitem{LC} C. Lu and X. Chen, \textit{Unicity theorem for generalized Gauss maps of immersed harmonic surfaces}, J. Math. Anal. Appl. \textbf{519} (2023), 126827.

\bibitem{M94} X. Mo, \textit{Value distribution of the Gauss map and the total curvature of complete minimal surface in $\R^m$}, Pacific. J. Math., \textbf{163}  (1994), 159--174.

\bibitem{No05} J. Noguchi, \textit{A note on entire pseudo-holomorphic curves and the proof of Cartan-Nochka's theorem}, Kodai Math. J. \textbf{28} (2005) 336--346

\bibitem{No81} J. Noguchi, \textit{Lemma on logarithmic derivatives and holomorphic curves in algebraic varieties}, Nagoya Math. J. \textbf{83} (1981), 213--233.

\bibitem{O59} R. Osserman, \textit{Proof of a conjecture of Nirenberg}, Comm. Pure Appl. Math., \textbf{12} (1959), 229--232.

\bibitem{O63}R. Osserman,  \textit{On complete minimal surfaces}, Arch. Rational Mech. Anal., \textbf{13} (1963), 392--404.

\bibitem{O64} R. Osserman, \textit{Global properties of minimal surfaces in $\E^3$ and $\E^n$}, Ann. of Math., \textbf{80} (1964), 340--364.

\bibitem{O86} R. Osserman,  A Survey of Minimal Surfaces, 2nd edition, Dover, New York, 1986.

\bibitem{QA} S. D. Quang and D. P. An, \textit{Second main theorem and unicity of meromorphic mappings for hypersurfaces in projective varieties}, Acta Math. Vietnam. \textbf{42} (2017), no. 3, 455--470.

\bibitem{QT} D. D. Thai and S. D. Quang, \textit{Non-integrated defect of meromorphic maps on K\"{a}hler manifolds}, Math. Z. \textbf{292} (2019), no. 1-2, 211--229.

\bibitem{QTT} S. D. Quang and D. D. Thai and P. D. Thoan, \textit{Distribution value of algebraic curves and the Gauss maps on algebraic minimal surfaces}, Intern. J. Math., \textbf{32} (2021), no. 05, 2150028, 13pp.

\bibitem{X81} F. Xavier, \textit{The Gauss map of a complete non-flat minimal surface cannot omit 1 points on the sphere}, Ann. of Math., \textbf{113} (1981), 211-214.

\bibitem{Y76} S.T. Yau, \textit{Some function-theoretic properties of complete Riemannian manifolds and their applications to geometry}, Indiana U. Math. J. \textbf{25} (1976), 659--670.
\end{thebibliography}
\end{document}